\documentclass[12pt]{amsart}

\usepackage{amsmath,graphicx}
\usepackage{adjustbox}
\usepackage{amssymb}
\usepackage{tikz, tikz-cd}
\usepackage{fullpage}
\usepackage[all]{xy}
\usepackage[margin=1in]{geometry}
\usepackage{amsthm}
\usepackage{amsfonts}
\usepackage{bbm}
\usepackage{eufrak}
\usepackage{comment}
\usepackage{amsmath}
\usepackage{amsthm}
\usepackage{bbm}
\usepackage{eufrak}
\usepackage{array} 
\usepackage{color}
\usepackage[utf8]{inputenc}
\usepackage[english]{babel}
\usepackage{hyperref}
\newcolumntype{L}{>{$}l<{$}}

\DeclareMathSymbol{\shortminus}{\mathbin}{AMSa}{"39}

\newtheorem{theorem}{Theorem}[section]
\newtheorem{lemma}[theorem]{Lemma}

\newtheorem{corollary}[theorem]{Corollary}

\newtheorem{proposition}[theorem]{Proposition}
\theoremstyle{definition}

\newtheorem{remark} [theorem] {Remark}

\theoremstyle{definition}

\newcommand{\C}{{\mathbb{C}}}
\newcommand{\F}{{\mathbb{F}}}

\newcommand{\Q}{{\mathbb{Q}}}

\newcommand{\Z}{{\mathbb{Z}}}

\newcommand{\lk}{\ell\text{k}}

\newcommand{\Sym}{\text{Sym}}

\DeclareMathOperator{\HFL}{HFL}
\DeclareMathOperator{\HFK}{HFK}

\DeclareMathOperator{\HF}{HF}

\DeclareMathOperator{\CKh}{CKh}
\DeclareMathOperator{\Kh}{Kh}

\DeclareMathOperator{\CFL}{CFL}
\DeclareMathOperator{\CFK}{CFK}

\DeclareMathOperator{\rank}{rank}
\DeclareMathOperator{\LLL}{L}

	\title{Rank Bounds in Link Floer Homology and Detection Results}
\author[Binns]{Fraser Binns}
\author[Dey]{Subhankar Dey}
\address[]{Department of Mathematics, Boston College}
\email{fb1673@princeton.edu}
\address[]{Department of Mathematics, University of Alabama}
\email{subhankar.dey@durham.ac.uk}

\begin{document}

	\maketitle
	
	\begin{abstract}
	Viewing the BRAID invariant as a generator of link Floer homology we generalise work of Baldwin-Vela-Vick to obtain rank bounds on the next-to-top grading of knot Floer homology. These allow us to classify links with knot Floer homology of rank at most eight, and prove a variant of a classification of links with Khovanov homology of low rank due to Xie-Zhang. In another direction we use a variant of Ozsv\'ath-Szab\'o's classification of $E_2$ collapsed $\Z\oplus\Z$ filtered chain complexes to show that knot Floer homology detects $T(2,8)$ and $T(2,10)$. Combining these techniques with the spectral sequences of Batson-Seed, Dowlin, and Lee we can show that Khovanov homology likewise detects $T(2,8)$ and $T(2,10)$.
	\end{abstract}

\begin{section}{Introduction}

Link Floer homology is a powerful invariant of links in $S^3$ due to Ozsv\'ath-Szab\'o, taking value in the category of multi-graded vector spaces~\cite{HolomorphicdiskslinkinvariantsandthemultivariableAlexanderpolynomial}. While powerful, it is not a \textit{complete} invariant; there exist non-isotopic links with isomorphic link Floer homology~\cite{hedden2018geography}. One of the main goals of this paper is to exhibit links which are fully determined -- ``detected" -- by their link Floer homology. More generally we address several variants of the ``geography" question for link Floer homology -- namely which vector spaces arise as the link Floer homology of some link -- as well as the ``botany" question -- which links  have a prescribed link Floer homology. 

One of our approaches to the geography problem goes through a result that uses a contact geometric invariant which lives in Heegaard Floer homology. Namely we study the \textit{BRAID} invariant, an invariant of transverse links due to Baldwin--Vela-Vick--V\'ertesi~\cite{baldwin2013equivalence} which is a version of Ozsv\'ath-Szab\'o's contact class~\cite{ozsvath2005heegaard}. In particular we generalize the rank bound from Baldwin--Vela-Vick~\cite[Theorem 1.1]{baldwin_note_2018}. The most concise version of our result is the following:

\newtheorem*{HFKrank}{Theorem~\ref{rankinnequalityforfiberedlinks}}
\begin{HFKrank}
	 Let $L$ be a non-trivial $n$ component fibered link in a rational homology sphere. Then $$\rank(\widehat{\HFK}(L;\dfrac{n-2-\chi(L)}{2}))\geq n.$$
\end{HFKrank}

Here $\chi(L)$ denotes the maximal Euler characteristic of a potentially disconnected surface bounding $L$. We may take coefficients to be $\Z,\Z_2$ -- the field with two elements -- or $\Q$, as we do for the remainder of this paper unless otherwise specified. Some care is required when working with $\Z$ or $\Q$, see sections~\ref{A Brief Review of Link Floer homology} and~\ref{BRAIDsection} for details. Note that $\dfrac{n-2-\chi(L)}{2}$ is the next to top Alexander grading of knot Floer homology. See Lemma~\ref{chainlevel} and Corollary~\ref{Nigeneralized} for stronger versions of Theorem~\ref{rankinnequalityforfiberedlinks} under appropriate additional hypotheses. 

Theorem~\ref{rankinnequalityforfiberedlinks} allows us to prove a number of detection results:

\newtheorem*{rank4}{Theorem~\ref{rank4}}
\begin{rank4}
If $\rank(\widehat{\HFK}(L))=4$ then $L$ is a Hopf link or the three component unlink.
\end{rank4}

\newtheorem*{rank6}{Theorem~\ref{rank6}}
\begin{rank6}
If $\rank(\widehat{\HFK}(L))=6$ then $L$ is a disjoint union of  an unknot and a trefoil.
\end{rank6}

\newtheorem*{rank8}{Theorem~\ref{rank8}}
\begin{rank8}
If $\rank(\widehat{\HFK}(L))=8$ then $L$ is $T(2,4)$, $T(2,-4)$, a four component unlink, or the disjoint union of a Hopf link and an unknot.
\end{rank8}

\newtheorem*{HFKdetcetstrefoilmeridian}{Theorem~\ref{HFKdetcetstrefoilmeridian}}

\begin{HFKdetcetstrefoilmeridian}
Knot Floer homology detects the link consisting of $T(2,3)$ and a meridian.
\end{HFKdetcetstrefoilmeridian}

These results can be viewed as extensions of some of the detection results in~\cite{BG} and hold for coefficients in $\Z_2, \Z$ or $\Q$.  We note also that while this manuscript was in preparation, an independent proof of the fact that the disjoint union of the Hopf link and the unknot is the only three component link with knot Floer homology of rank eight appeared in~\cite{kim2020links}.

Applying Dowlin's spectral sequence from Khovanov homology to knot Floer homology, Theorems~\ref{rank4}, \ref{rank6} and \ref{rank8} allow us to prove variants of results due to Xie-Zhang~\cite{xie2022links}. 

\newtheorem*{KHr2components}{Corollary~\ref{KHr2components}}
\begin{KHr2components}
Suppose that $L$ is a two component pointed link with $\rank(\widetilde{\Kh}(L,p;\Q))\leq 4$. Then $L$ is one of the following: \begin{itemize}
    \item an unlink;
\item a Hopf link;
\item $T(2,4)$ or $T(2,-4)$.
\end{itemize}
\end{KHr2components}

\newtheorem*{KHr3components}{Corollary~\ref{KHr3components}}
\begin{KHr3components}
Suppose that $L$ is a three component pointed link. Then $\rank(\widetilde{\Kh}(L,p;\Q)) >2$.
\end{KHr3components}
Here $\widetilde{\Kh}(L,p;\Q)$ is the reduced Khovanov homology of the pointed link $(L,p)$. Tye Lidman has pointed out to the authors that Corollary~\ref{KHr3components} can be obtained from the spectral sequence from reduced Khovanov homology of $L$ to the double branched cover of $L$ due to Ozsv\'ath-Szab\'o~\cite{ozsvath2005heegaard}.

Using different methods, we can show link Floer homology detection results for a number of infinite families of links. Let $K_{2,2n}$ denote the cable link with pattern $T(2,2n)$ and companion $K$. We show that link Floer homology detects the $(2,2n)$ cables of two of the non-trivial fibered knots that knot Floer homology is currently known to detect, namely the trefoils and the figure eight\footnote{Since this paper first appeared, it has also been shown that knot Floer homology detects $T(2,\pm 5)$~\cite{farber2022fixed} as well as some other knots, see~\cite{baldwin2022floer}}.

\newtheorem*{HFL2detectsRHTcables}{Theorem~\ref{HFL2detectsRHTcables}}
\begin{HFL2detectsRHTcables}
Link Floer homology detects $T(2,3)_{2,2n}$ for all $n$.
\end{HFL2detectsRHTcables}

Note here that the $(2,2n)$-cable of $T(2,3)$ is the $(2,-2n)$-cable of $T(2,-3)$.

\newtheorem*{HFLdetects(2,2n)cablesofthefigureeight}{Theorem~\ref{HFLdetects(2,2n)cablesofthefigureeight}}
\begin{HFLdetects(2,2n)cablesofthefigureeight}
Let $K$ be the figure eight knot i.e. $4_1$. Link Floer homology detects $K_{2,2n}$ for every $n$.
\end{HFLdetects(2,2n)cablesofthefigureeight}
We prove these results  over $\Q,\Z$ and $\Z_2$, apart from in the $n=0$ cases in which case we only prove the result with $\Z_2$ coefficients.

Our second approach to the geography problem is via algebra, namely a version of Ozsv\'ath-Szab\'o's classification of $E_2$ collapsed chain complexes~\cite[Section 12]{HolomorphicdiskslinkinvariantsandthemultivariableAlexanderpolynomial}. We obtain the following results:

\newtheorem*{HFKT28}{Theorem~\ref{HFKT28}}
\begin{HFKT28}
Knot Floer homology detects $T(2,8)$.
\end{HFKT28}
\newtheorem*{HFKT210}{Theorem~\ref{HFKT210}}
\begin{HFKT210}
Knot Floer homology detects $T(2,10)$.
\end{HFKT210}

These results again hold over $\Q,\Z$ and $\Z_2$. Here we have given $T(2,8)$ and $T(2,10)$ the braid orientation, contrary to the case in~\cite{binns2020knot}, where knot Floer homology was shown to detect $T(2,2n)$, oriented as the boundary of an annulus, for all $n$.

Combining these techniques with the Dowlin~\cite{dowlin2018spectral}, Lee~\cite{lee2005endomorphism}, and Batson-Seed~\cite{batson2015link} spectral sequences we obtain the corresponding detection results for Khovanov homology.

\newtheorem*{KHT28}{Theorem~\ref{KHTT28}}
\begin{KHT28}
Suppose $\Kh(L;\Z)\cong\Kh(T(2,8);\Z)$. Then $L$ is isotopic to $T(2,8)$.
\end{KHT28}
\newtheorem*{KHT210}{Theorem~\ref{KHT210}}
\begin{KHT210}
Suppose $\Kh(L;\Z)\cong\Kh(T(2,10);\Z)$. Then $L$ is isotopic to $T(2,10)$.
\end{KHT210}
Khovanov homology was previously known to detect $T(2,2n)$ for $n=\pm 3$ by work of Martin~\cite{martin2022khovanov}, $n=\pm 2$ by Xie-Zhang~\cite{xie_instanton_2019}, $n=\pm 1$ by Baldwin-Sivek-Xie~\cite{baldwin2019khovanov}, and $n=0$ by a combination of Hedden-Watson~\cite{hedden2010manifolds} and Kronheimer-Mrowka~\cite{kronheimer2011khovanov}.

	The outline of this paper is as follows: in Sections~\ref{A Brief Review of Link Floer homology} and~\ref{BRAIDsection} we review pertinent properties of link Floer homology and the BRAID invariant, respectively. In Section~\ref{linkcontact} we use the BRAID invariant to prove several rank bounds results in link Floer homology. We collect our detection results for $(2,2n)$-cables in Section~\ref{Detection Results for $(2,2n)$-cables}. We prove a technical results necessary for subsequent Sections in Section~\ref{Technical Results}. Section~\ref{Knot Floer Homology Detects $T(2,8),T(2,10)$} is devoted to showing that knot Floer homology detects $T(2,8)$ and $T(2,10)$, in Section~\ref{Rank Detection Results for Knot Floer homology} we give botany results for the rank of Knot Floer homology, and in Section~ \ref{Detection Results for Khovanov Homology} we give detection results for Khovanov homology.

	\subsection*{Acknowledgements}
We would like to thank the organisers of the ``Trends in Low Dimensional Topology" seminar series for providing the platform which led to this project, as well as Tye Lidman and Jen Hom for comments on an earlier draft. We are especially grateful to Lev Tovstopyat-Nelip for pointing out an error in an earlier version of this paper, and the referee for helping us make numerous significant improvements. The first author would also like to thank John Baldwin and Gage Martin for numerous helpful conversations, as well as Tao Li. The second author would like to acknowledge partial support of NSF grants DMS 2144363, DMS 2105525, and AMS-Simons travel grant.
	\end{section}

	\begin{section}{A Very Brief Review of Link Floer homology}\label{A Brief Review of Link Floer homology}
	
	In this section we review link Floer homology, partially to fix conventions and notation, with an emphasis on the structural properties that we will use in the subsequent sections.
	
Let $L$ be an oriented $n$ component link in a rational homology sphere $M$. The pair $(L,M)$ can be encoded as a Heegaard diagram for $M$ with $n$ pairs of basepoints $\{w_i,z_i\}$. The link Floer complex of $L$, defined by Ozsv\'ath-Szab\'o in~\cite{HolomorphicdiskslinkinvariantsandthemultivariableAlexanderpolynomial}, is a multi-graded $\Z_2[U_1,U_2,\dots U_n]$ chain complex with underlying vector space: $$\CFL^-(L,M)\cong\bigoplus_{m,A_1,A_2,\dots A_n}\CFL^-_m(L,M)(A_1,A_2\dots A_n).$$ 

The three manifold $M$ is often apparent from the context and duly suppressed in the notation. The $A_i$ gradings are called the \emph{Alexander gradings} and can be thought of as elements in $\Z+\frac{\lk(L_i,L-L_i)}{2}$, while $m$ is called the \emph{Maslov grading}, which is integer valued. $\CFL^-(L)$ is endowed with a differential which counts pseudo-holomorphic disks in a certain auxiliary symplectic manifold endowed with an appropriate almost complex structure which do not intersect submanifolds $V_{z_i}$ determined by $z_i$. These counts are weighted by a count of intersections with the manifolds $V_{w_i}$, determined by $w_i$. The filtered chain homotopy type of $\CFL^-(L)$ is an invariant of $L$, which we refer to as the \emph{link Floer complex of $L$}. The homology of $\CFL^-(L)$ is denoted $\HFL^-(L)$. $\CFL^-(L)$ has a quotient defined by setting $U_i=0$ for all $i$ whose filtered chain homotopy type is also an invariant of $L$. This complex is denoted by $\widehat{\CFL}(L)$, and has homology denoted by $\widehat{\HFL}(L)$, which we will call the \emph{link Floer homology of $L$}. The \emph{link Floer homology polytope} of $L$ is defined as the convex hull of the multi-Alexander gradings of $\widehat{\HFL}(L)$ with non-trivial support. The link Floer polytope determines the Thurston polytope of the link exterior by a result of Ozsv\'ath-Szab\'o~\cite{ozsvath2008linkFloerThurstonnorm}. Note that $\widehat{\HFL}(\overline{L})$ is the dual -- in the sense of ~\cite[Equation 3]{BG} -- of $\widehat{\HFL}(L)$, where $\overline{L}$ is the mirror of $L$.

For each component $L_i$ of $L$ there is a spectral sequence from $\widehat{\HFL}(L)$ to $\widehat{\HFL}(L-L_i)\otimes V[\frac{\lk(L_i,L_j)}{2}]$. Here $V\cong \F_0\oplus\F_{-1}$, supported in multi-Alexander grading zero, while $[\frac{\lk(L_i,L_j)}{2}]$ indicates a shift in the $A_j$ grading, corresponding to $L_j$ which is a component in $L$ other than $L_i$. Indeed, $\widehat{\HFL}(L)$ can be viewed as the graded part of a multi-filtered chain complex with total homology $\widehat{\HF}(M)\otimes V^{n-1}$, where $\widehat{\HF}(M)$ is the Heegaard Floer homology of $M$ as defined by Ozsv\'ath-Szab\'o~\cite{ozsvath2004holomorphic}. If a component $L_i$ of $L$ is fibered and $L$ is non-split, then $L-L_i$ is braided about $L_i$ exactly when $\widehat{\HFL}(L)$ is of rank $2^{n-1}$ in the maximal $A_i$ grading of non-trivial support by a result of Martin~\cite{martin2022khovanov}.
	 
	 The \emph{knot Floer homology} of $L$, $\widehat{\HFK}(L)$, is an oriented link invariant due independently to J. Rasmussen~\cite{Rasmussen} and Ozsv\'ath-Szab\'o \cite{Holomorphicdisksandknotinvariants} that can be obtained from $\widehat{\HFL}(L)$ by projecting the multi-Alexander grading onto the diagonal, and shifting the Maslov gradings up by $\frac{n-1}{2}$. For an $n$ component link $L$, $\widehat{\HFK}(L)$ can also be thought as the knot Floer homology of the knotified link $\kappa(L) \subset(\#^{n-1}(S^2 \times S^1))\#M$ \cite{ozsvath2004holomorphic}. Since the knot Floer chain complex is a \textit{filtered} version of Heegaard Floer chain complex, we have that $\dim(\widehat{HFK}(L)) \geq \dim(\widehat{\HF}(M(\#^{n-1}(S^2 \times S^1))) = 2^{n-1}\dim(\widehat{\HF}(M))$. It likewise follows that $\dim(\widehat{\HFK}(L))\equiv \dim(\widehat{\HF}(M))\cdot 2^{n-1}\mod{2}$, so that $\dim(\widehat{HFK}(L))$ is odd only if $L$ has a single component. Also note that $\widehat{\HFK}(L)$ determines the maximal Euler characteristic of a surface bounding $L$ \cite{ni2006note}, as well as whether or not $L$ fibered~\cite{ghiggini2008knot},~\cite{ni2007knot}. 
	 
	 Versions of $\widehat{\HFL}(L,M)$ and $\widehat{\HFK}(L,M)$ can also be defined with $\Z$ or $\Q$ coefficients~\cite{sarkar2011note}. In fact there are $2^{n-1}$ versions of each theory, each corresponding to a \emph{coherent system of orientations} on the moduli space of pseudo-holomorphic disks. We will only have to be careful with which version of these theories we are using when applying work of Dowlin~\cite{dowlin2018spectral} which gives a spectral sequence from Khovanov homology to Knot Floer homology endowed with the orientation given in~\cite{alishahi2015refinement}. As is customary, we suppress the dependence of knot Floer homology on the coherent system of orientations in our notation. Each of the knot Floer homology theories detects the maximal Euler characteristic of a surface bounding $L$~\cite{ni2006note}, as well as whether or not $L$ is fibered. This follows from work of Juh\'asz~\cite{juhasz2008floer}.
	 
	 The \emph{Conway polynomial} of $L$ can be obtained as an appropriate decategorification of $\widehat{\HFL}(L)$. A result of Hoste~\cite{hoste1985firstcoefficientoftheconwaypolynomial} implies that the Conway polynomial detects the linking number of two component links. It follows that knot Floer homology and link Floer homology also detect the linking number of two component links.

	\end{section}
	
	\begin{section}{The BRAID Invariant}\label{BRAIDsection}
	In this section we review the BRAID invariant, an invariant of (generalized) braids due to Baldwin-Vela-Vick-V\'ertesi~\cite{baldwin2013equivalence}. We find two different perspectives on braids to be useful. First we  define a braid in an open book $(\Sigma,\phi)$ to be a link which intersects every page transversely, up to isotopy through families of such links. Any such braid can be recovered from the data of a \emph{pointed open book} $(\Sigma,\phi,\{p_i\})$, where $\phi:\Sigma\to\Sigma$ fixes the points $\{p_i\}\subseteq\Sigma$ as a set. We assume henceforth that $(\Sigma,\phi)$ is an open book decomposition for some rational homology sphere.

	The BRAID invariant of a braid $L$ represented by a pointed open book $(\Sigma,\phi,\{p_i\})$ is defined concretely from a Heegaard diagram determined by the pointed book. We recall the construction here. Take a \emph{basis of arcs} $\{a_j\}$ for $\Sigma-\{p_i\}$. Here a basis of arcs is a maximal collection of properly embedded arcs in $\Sigma-\{p_i\}$ that are homologically independent in $H_1(\Sigma-\{p_i\},\partial\Sigma)$. An example of such is shown in red in the lower half of the surface shown in Figure~\ref{BRAIDDISK}. Consider the basis of arcs $b_j$ obtained by pushing the ends of the $a_j$ curves around $\partial \Sigma$ in the direction dictated by the orientation of $\partial\Sigma$ and isotoping so that $a_j\cap b_j$ consists of a single point, $c_j$. In Figure~\ref{BRAIDDISK} the points $c_j$ are represented by green dots. Form a multi-pointed Heegaard diagram consisting of $\Sigma\cup_{\partial\Sigma}(-\Sigma)$ with $\alpha$ curves consisting of the union of the $a$ arcs in $\Sigma$ and $-\Sigma$, and $\beta$ curves consisting of the union of the $b$ arcs in $\Sigma$ and $\phi(b)$ in $-\Sigma$. In Figure~\ref{BRAIDDISK} we have not shown $\phi(b_j)$ for all $j$ to avoid clutter. The $w_i$ basepoints consist of $\{p_i\} \subset-\Sigma$ while the $z_i$ base points consist of $\{p_i\}\subset\Sigma$. In Figure~\ref{BRAIDDISK} the two black dots in the upper half of the diagram are $w_i$ basepoints, while the two black dots in the lower half of the diagram represent $z_i$ basepoints. We call diagrams constructed in this manner Heegaard diagrams \emph{adapted} to $(\Sigma,\phi,\{p_i\})$, $\{a_j\}$. The BRAID invariant of $L$ is defined as the homology class of $\{c_i\}$. This can be viewed as a class in either $\widehat{\CFL}(-M,L;\Z_2)$, or $\CFL^-(-M,L;\Z_2)$. We denote these classes by $\widehat{t}(L)$ or $t(L)$ respectively.  Baldwin-Vela-Vick-V\'ertesi show that $\widehat{t}(L)$ and $t(L)$ are braid invariants, when viewed as elements of $\widehat{\HFK}(-M,L;\Z_2)$ and $\HFK^-(-M,L;\Z_2)$ respectively. Tovstopyat-Nelip further notes that the BRAID invariant can be thought of as a well defined class in $\widetilde{\HFK}(-M,L;\Z_2)$, a version of knot Floer homology defined for links encoded by multiple basepoints~\cite{tovstopyat2018transverse}. We note that $\widehat{t}(L)$ is in fact well defined in $\widehat{\HFL}(-M,L;\Z_2)$ complete with all Alexander gradings. We note also that we may conflate $\widehat{t}(L)$ and the associated generator of the chain complex in the section~\ref{linkcontact}.

\begin{figure}[h]
    \centering
    \includegraphics[width=7cm]{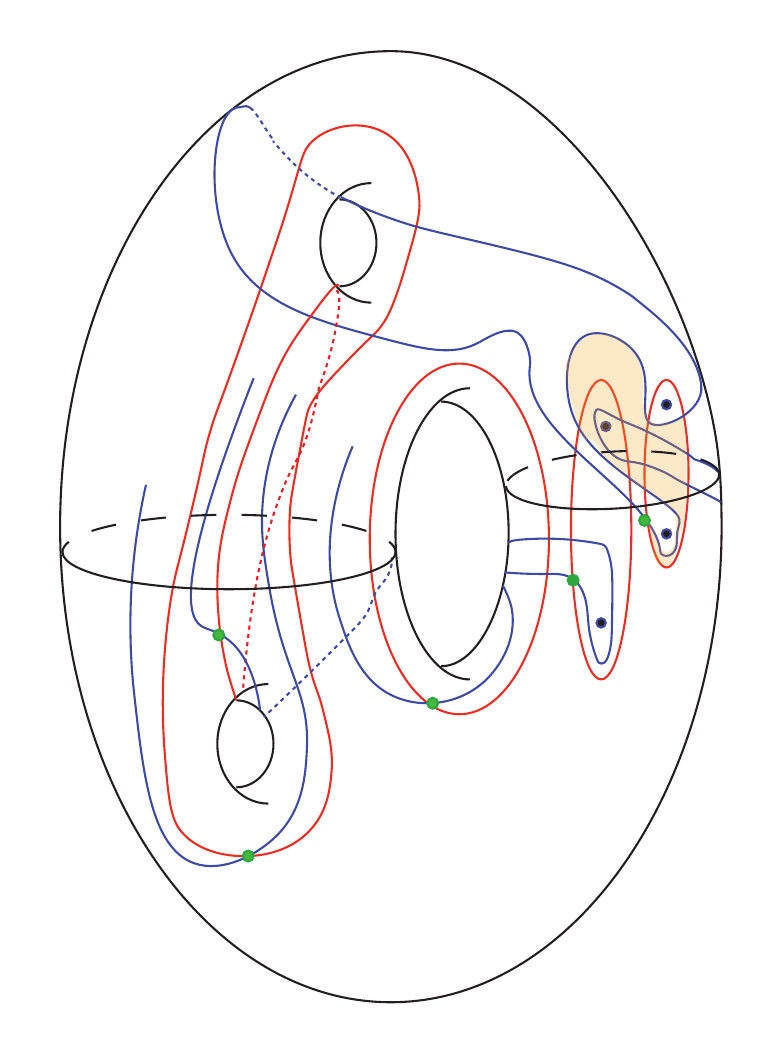}
    \caption{A Heegaard diagram adapted to a pointed open book $(\Sigma,\phi,\{p_i\})$ and basis of arcs $\{a_j\}$. Note that we have only shown segments of some of the $\beta$ curves. The green dots indicate the intersection points that represent the BRAID invariant. The orange region is the shadow of a pseudo-holomorphic disk which plays a role in Section~\ref{linkcontact}. The blue dots in the lower half of the surface are $w_i$ basepoints, which come from the points $p_i$ in the pointed open book. The basepoints in the upper half of the surface are $z_i$.  Note that not all of the basepoints are shown in the figure.
    }\label{BRAIDDISK}
    \end{figure}
	
	The BRAID invariant has some strong non-vanishing properties. The strongest of these, a result due to Tovstopyat-Nelip, is that $\widehat{t}(B\cup K)\neq 0$ where $B$ is a link braided about a fibered link $K$~\cite[Theorem 1.3]{tovstopyat2018transverse}. Note that for $\widehat{t}(B\cup K)$ to be well defined, we implicitly push $K$ transversely off itself, so that $B\cup K$ is braided with respect to $(\Sigma,\phi)$. Note that Tovstopyat-Nelip's result generalises earlier work of Vela-Vick~\cite{vela2011transverse} and Vela-Vick-Etnyre~\cite{etnyre2010torsion}. In the case that $B=\emptyset$, the Alexander grading of $\widehat{t}(B\cup K)$ is readily computed~\cite[Lemma 4.2]{tovstopyat2018transverse}. 
	  
 We now discuss some mapping-class group theoretic properties of braids which will be of importance in Section~\ref{linkcontact}. Given a pointed open book $(\Sigma,\phi,\{p_i\})$ there are two notions of an arc being sent to the right. We introduce a new terminology to distinguish these two notions. We say an arc $a\subseteq\Sigma$ is \emph{relatively sent to the right} by $\phi$ if after $\phi(a)$ has been isotoped in $\Sigma-\{p_i\}$ so that $\phi(a)$ and $a$ intersect minimally, $\phi(a)$ is to the right of $a$. This is the notion of being sent to the right used in~\cite{baldwin2013equivalence}. Note that if a pointed open book $(\Sigma,\phi,\{p_i\})$ representing a braid $L$ has an \textit{essential} arc which is relatively sent to the left then $\widehat{t}(L)$ and $t(L)$ are both trivial in the homology (\cite[Theorem 1.4]{baldwin2013equivalence}). We say that a pointed open book is \emph{relatively right veering} on a component of $\partial\Sigma$ if every arc which intersects that component of $\partial\Sigma$ is relatively sent to the right.
 
 We say that $a$ is \emph{absolutely sent to the right} if $a$ is sent to the right after isotoping $\phi(a)$ in $\Sigma$ -- in particular $\phi(a)$ is allowed to pass over basepoints under the isotopy. We say that a pointed open book is \emph{absolutely right veering} on a component of $\partial\Sigma$ if every arc which intersects that component of $\partial\Sigma$ is absolutely sent to the right. The distinction between these two concepts plays an important role in the next section. Similar definitions also apply for notions of left veering.

 For the purposes of this paper we do not count arcs that are fixed up to isotopy as right or left veering. That is we have the following trichotomy: arcs are either (relatively) sent to the left, right, or fixed up to isotopy.
 
 We conclude this section by noting that, while the braid invariant was originally defined in homology theories with coefficients in $\Z_2$, it is in fact well defined up to sign with coefficients in $\Z$. In particular, if $L$ is a fibered link, the BRAID invariant generates the top Alexander grading of $\widehat{\HFK}(L;\Z)$ or $\widehat{\HFK}(L;\Q)$.
	\end{section}

\begin{section}{A Rank Bound from the BRAID invariant}\label{linkcontact}

	Baldwin-Vela-Vick proved that the knot Floer homology of a fibered knot is non-trivial in the next to top Alexander grading \cite{baldwin_note_2018}. In this section we generalise their techniques to obtain a number of related results. In particular we show:
	
	\begin{theorem}\label{rankinnequalityforfiberedlinks}
	 Let $L$ be a non-trivial $n$ component fibered link in a rational homology sphere,  then: \[\rank(\widehat{\HFK}(L,\dfrac{n-2-\chi(L)}{2}))\geq n.\]
	\end{theorem}
	
	This implies a generalization of a Theorem of Ni~\cite[Theorem A.1]{ni2020exceptional}:
	
	\begin{corollary}\label{Nigeneralized}
	Suppose $L$ is a non-trivial fibered $n$-component link in a rational homology sphere, with monodromy neither relatively right veering nor relatively left veering on each component. Then \[\rank(\widehat{\HFK}(L,\dfrac{n-2-\chi(L)}{2}))\geq 2n.\]
	\end{corollary}

Before proceeding to the technical lemma which underlies these results, we fix some notation. $\partial_T$ shall denote the component of the total differential on $\widehat{\CFL}(L)$, namely the one for which $H_*(\widehat{\CFL}(L),\partial_T)\cong\widehat{\HF}(M)\otimes V^{n-1}$, where $M$ is the underlying rational homology sphere and $n$ is the number of components of $L$. $\partial_{\widehat{\CFL}}$ will denote the link Floer Homology differential; i.e. $H_*(\widehat{\CFL}(L),\partial_{\widehat{\CFL}})\cong\widehat{\HFL}(L)$.

\begin{lemma}\label{chainlevel}
    Suppose $(\Sigma,\phi,\{p_i\})$ is a pointed open book representing a braid $L$ in a rational homology sphere $M$. Suppose $\phi$ sends an essential arc $a\subset(\Sigma,\{p_i\})$ absolutely to the left. Then there is a generator $\mathbf{d}\in\widehat{\CFL}(-M,\overline{L},[\Sigma])$ with 
   $\partial_T\mathbf{d}=\widehat{t}(L)$ 
    and $\partial_{\widehat{\CFL}}(\mathbf{d})=0$. 
\end{lemma}

To prove this Lemma~\ref{chainlevel} we explicitly find the generator $\mathbf{d}$. We then show that $\partial_T\mathbf{d}=\widehat{t}(L)$ using a diagrammatic argument, and note that it is easy to determine the multi-Alexander grading of $\mathbf{d}$ relative to that of $\hat{t}(L)$.

\begin{proof}[Proof of Lemma~\ref{chainlevel}]

Fixing notation as in the statement of the lemma, we extend $a$ to a basis of arcs $\{a_j\}$ for $(\Sigma,\phi,\{p_i\})$. Then we isotope $\phi(a)$ to a curve $b$ in $-\Sigma$ so that $a$ and $b$ intersect minimally. Consider the surface $\Sigma\cup-\Sigma$, and the alpha and beta curves -- $\alpha,\beta$ -- associated to $a$ in the Heegaard diagram adapted to $((\Sigma,\phi)$, $\{a_j\})$. Note that since $a$ is absolutely sent to the left there is an intersection point $d\in\alpha\cap \beta\cap-\Sigma$ and a bigon in $\Sigma\cup(-\Sigma)$ with corners at $c\in\alpha\cap\beta,\in\widehat{t}(L)$ and $d$ and edges contained in $\alpha$ and $\beta$.

Consider now $b'$, the arc formed after isotoping $\phi(a)$ in $(-\Sigma,\{p_i\})$ so that $a$ and $b'$ intersect minimally. Let $\beta'$ be the associated beta curve in the adapted Heegaard diagram. We claim that there is a map $B:\{z\in\C:|z|\leq 1\}\to \Sigma\cup-\Sigma$ with $B(\{z:|z|=1,\Re(z)\leq 0\})\subseteq \beta'$, $B(\{z:|z|=1,\Re(z)\geq 0\})\subseteq \alpha$ whose image, counted with multiplicity, is a linear combination of regions $(\Sigma\cup-\Sigma)\backslash(\cup_i(\alpha_i)\cup\cup_i(\beta'_i))$. To see this let $b_t:[0,1]\times[0,1]\to\Sigma$ be an isotopy from the arc $b_0=b$ to the arc $b_1=b'$ (where we view arcs as smooth embedding $b,b':[0,1]\to\Sigma$). Note that there is a continuous map $[0,1]\to\{0,1\}$ which takes the value $0$ if the image of $b_t$, counted with multiplicity is a linear combination of regions $(\Sigma\cup-\Sigma)\backslash(\alpha\cup\beta_t)$, where $\beta_t$ are the $\beta$ curves obtained from $b_t$. Examples of such maps of bigons can be seen shaded in orange in Figure~\ref{BRAIDDISK}, or in Figure~\ref{fiberedBRAID}. The claim follows from continuity. We note that the image of $B$ may have multiplicity strictly greater than one at certain points.

Let $d'=B(i)\in-\Sigma$, as shown in Figure~\ref{fiberedBRAID}. Consider the generator of $\widehat{\CFL}(-M, \overline{L})$, $\mathbf{d}$, given by $\{d',c_2,\dots c_{2g+n-1}\}$. Here the intersection points $c_i$ are intersection points of the remaining $\alpha_i$ and $\beta_i$ curves in $\Sigma$. $\langle\partial_T\mathbf{d},\widehat{t}(L)\rangle\neq 0$, as witnessed by the bigon $B$, which is the shadow of a pseudo-holomorphic disk in $\Sym^{|\alpha|}(\Sigma\cup-\Sigma)$, where $|\alpha|$ is the number of $\alpha$ curves. There is no other pseudo-holomorphic disk from $\mathbf{d}$ to $\widehat{t}(L)$ -- emanating to the north east of $d'$ as opposed to south west, in Figure~\ref{BRAIDDISK}, or in Figure~\ref{fiberedBRAID} -- as this would imply that $a$ and $\phi(a)$ were isotopic, contradicting the hypothesis that $a$ is absolutely sent the left.

Finally $\partial_{\widehat{\CFL}} \mathbf{d}=0$, as else $a$ and $b'$ would not intersect minimally as arcs in $(-\Sigma,\{p_i\})$.

\end{proof}

We note that in Lemma~\ref{chainlevel}, the multi-Alexander grading of $\mathbf{d}$ is determined by the multi-Alexander grading of $\hat{t}(L)$ and the number of $w_i$ basepoints and their respective multiplicity, $n_i$, in the image of the map $B$. Let $x_k$ denote the $k$th Alexander grading of the BRAID invariant of a braid $L$. It follows that $\mathbf{d}$ has $A_k$ grading given by $x_k+n_k$ for all $k$.

\begin{figure}[h]
    \centering
    \includegraphics[width=6cm]{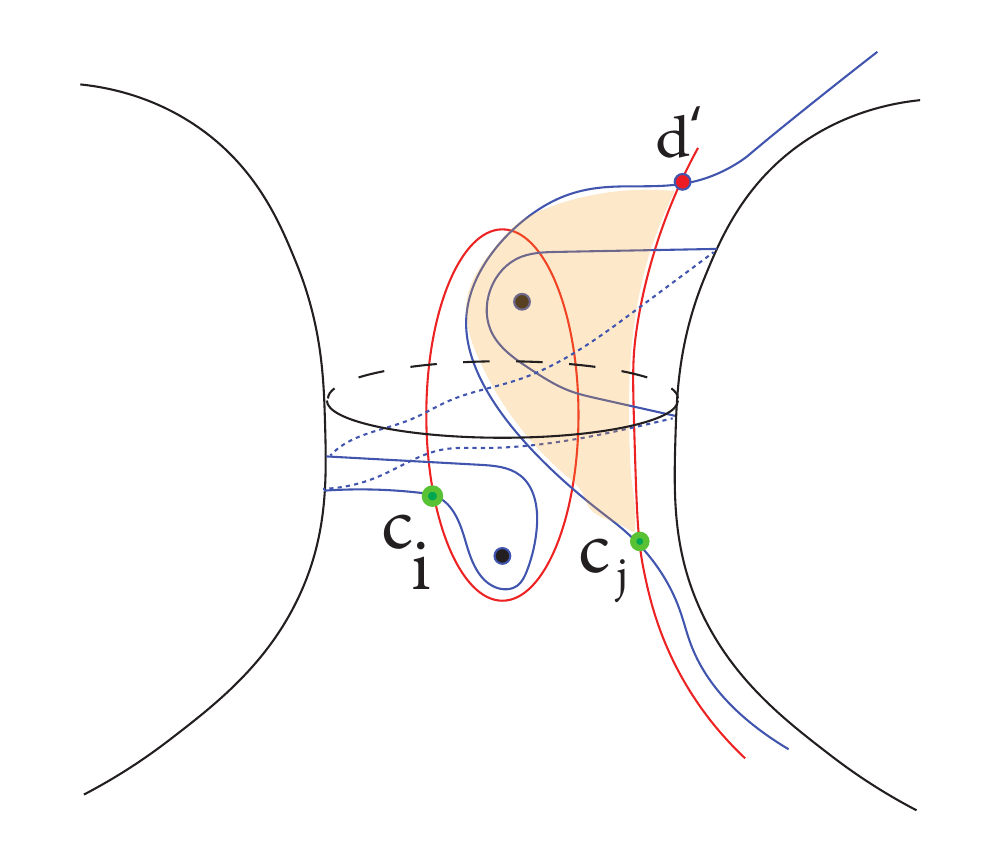}
    \caption{A picture of part of the Heegaard diagram used to define the BRAID invariant for the binding of an open book, via a transverse push-off. The BRAID invariant includes the green dots $c_i$ and $ c_j$. As described in Lemma~\ref{generalnonzero}, $\mathbf{d}'$ includes $c_i$, but not $c_j$, instead including $d$. The orange region is the shadow of a pseudo-holomorphic disk from $\mathbf{d}$ to the BRAID invariant. The black dots are basepoints.}\label{fiberedBRAID}
    \end{figure}

We now prove a more general version of Theorem~\ref{rankinnequalityforfiberedlinks}.
\begin{lemma}\label{generalnonzero}
  With the hypotheses of Lemma~\ref{chainlevel} together with the additional hypothesis that $\widehat{t}(L)\neq 0$ we have that $\rank(\bigoplus_{\mathbf{x}\in I}\widehat{\HFL}(L,\mathbf{x}))\geq 1$ where here: $$I=\{(A_1,A_2,\dots, A_{n})| x_j\leq A_j\leq x_j+n_j, \text{for all $j$}, \text{ and }\exists i\text{ such that } A_i\neq x_i\}.$$
  
  and $n$ is the number of components of $L$ and we take coefficients in a field.
\end{lemma}

\begin{proof}
    We proceed using the same notation as in Lemma~\ref{chainlevel}. Note that $\mathbf{d}$ is a cycle in $(\widehat{\CFL}(L),\partial_{\CFL})$. If $\mathbf{d}\neq \partial_{\widehat{\CFL}}\mathbf{e}$ for any $\mathbf{e}$, the result follows. Suppose then that there is an $\mathbf{e}$ with $\partial_{\widehat{\CFL}}\mathbf{e}=\mathbf{d}$.
    
   Observe that since we are working with coefficients in a field $(\widehat{\CFL}(L),\partial_{T})$ has a model in which $\widehat{\CFL}(L)\cong\widehat{\HFL}(L)$ (See \cite[Reduction Lemma]{hedden2018geography}) in every Alexander grading except for that of $\mathbf{d}$, $\widehat{t}(L)$. Since $\partial_T(\mathbf{d})=\widehat{t}(L)\neq 0$ and  $\partial_T^2\mathbf{e}=0$ 
    we have that $\partial_T(\mathbf{e})-\mathbf{d}\neq 0\in\bigoplus_{\mathbf{x}\in J}\widehat{\CFL}(L,\mathbf{x})$ where $$J=\{(A_1,A_2,\dots, A_{n})| x_j\leq A_j\leq x_j+n_j,  \text{for all $j$, and }\exists i\text{ such that } A_i\neq x_i+n_i\}.$$
    
    Suppose towards a contradiction that $\rank(\bigoplus_{\mathbf{x}\in I}\widehat{\HFL}(L,\mathbf{x}))=0$. Then in fact we have that $\partial_T(\mathbf{e})-\mathbf{d}\in\widehat{\HFL}(L,(x_1,x_2,\dots x_n))$. It follows that $\partial_T(\partial_T(\mathbf{e})-\mathbf{d})=\partial_{\widehat{\CFL}}(\partial_T(\mathbf{e})-\mathbf{d})=\widehat{t}(L)$, contradicting the fact that $\widehat{t}(L)\neq 0$.
\end{proof}

\begin{remark}
Suppose $(\Sigma,\phi,\{p_i\})$ is a pointed open book for a link $L$ in a rational homology sphere $M$ such that neither $\phi$ nor $\phi^{-1}$ sends an essential arc absolutely to the left. Then $\phi$ fixes every essential arc and is in fact the identity. In this case, $L$ is an unlink in $M$. The assumption on the arcs used in Lemma~\ref{chainlevel} and Lemma~\ref{generalnonzero} is therefore only the mildest of restrictions, at least up to mirroring.

\end{remark}

We now specialise Lemma~\ref{generalnonzero} to prove Theorem~\ref{rankinnequalityforfiberedlinks}.

\begin{proof}[Proof of Theorem~\ref{rankinnequalityforfiberedlinks}]

An $n$ component fibered link $L$ arising as the binding of an open book $(\Sigma,\phi)$ may be perturbed to yield an $n$-braid $B$ in the complement of $L$, with associated pointed monodromy $(\Sigma,\phi')$. Consider the $m$th boundary component of $\Sigma$, $P_m$. After mirroring we may assume that $\phi$ sends some non-separating arc $a$ with an endpoint on $P_m$ absolutely to the left or fixes some non-separating arc. Extend this arc to a basis for $(\Sigma,\phi')$, and consider $\widehat{t}(B)$.

Suppose $a$ is fixed up to isotopy. Then $M$ contains an $S^1\times S^2$ summand and hence cannot be a rational homology sphere, a contradiction.

Suppose $a$ is sent absolutely to the left. Then $\widehat{t}(B)\neq 0$ by~\cite[Theorem 1.1]{tovstopyat2018transverse}. Suppose $\widehat{t}(B)$ has multi-Alexander grading $(A_1,A_2,\dots, A_{n})$. The bigon obtained via~\ref{chainlevel}, and pictured in Figure~\ref{fiberedBRAID} in this special case, contains a single $w_m$ basepoint. Thus applying Lemma~\ref{generalnonzero}, and noting that in this special case the hypothesis that we work with a field is unnecessary,  yields a non-trivial generator in multi-Alexander grading $(A_1,A_2,\dots, A_{m}-1,\dots A_{n})$.

Applying this procedure for each boundary component yields a total of $n$ distinct generators with Alexander-multi grading $(A_1,A_2,\dots, A_n)$ satisfying $A_1+A_2+\dots A_n=\dfrac{n-2-\chi(L)}{2}$, whence the result follows from the fact that knot Floer homology is obtained by collapsing the multi-Alexander grading to a single grading.
\end{proof}

\begin{proof}[Proof of Corollary~\ref{Nigeneralized}]
 View $L$ as an $n$-braid as in the proof of Theorem~\ref{rankinnequalityforfiberedlinks}. Suppose $\widehat{t}(L)$ is of multi-Alexander grading $(A_1,A_2,\dots A_n)$. Theorem~\ref{rankinnequalityforfiberedlinks} implies that for all $i$: $$\rank(\widehat{\HFL}(L;(A_1,\dots A_i-1,\dots A_n)))\geq1.$$ Suppose there exists an $i$ for which $\rank(\widehat{\HFL}(L;(A_1,\dots A_i-1,\dots A_n)))=1$. Consider $\CFL^-(L)$, together with the differential $\partial^{-}_T$ which counts all pseudo-holomorphic disks. Note that $\CFL^-(L)$ can be thought of as a $((A_1,A_2,\dots A_n),(j_1,j_2,\dots j_n))$-graded complex where the $j_i$ gradings are the $U_i$ gradings, so that the action of $U_i$ decreases the $j_i$ and $A_i$ grading by $1$ and preserves the other gradings. Since $L$ is fibered, there is a unique generator $x$ of $\CFL^{-}(L)$ with maximal $A_1+A_2+\dots +A_n$ grading and $j_i=0$ for all $i$. Since $\rank(\widehat{\HFL}(L;(A_1,\dots A_i-1,\dots A_n)))=1$, we may take there to be a unique generator $y$ of $\CFL^-(L)$ with Alexander multi-grading $(A_1,A_2,\dots A_i-1,\dots,A_n)$, and again $j_i=0$ for all $i$. 
 
 Since $\phi$ sends an arc in $\Sigma$ with endpoint on $L_i$ to the left there is a component of $\partial_T^-$ from $x$ to $y$. Note that this property is preserved by filtered chain homotopy. Since $\phi$ sends an arc in $\Sigma$ with endpoint on $L_i$ to the right there is a component of $\partial_T^-$ from $y$ to $U_ix$. Note again that this property is preserved by filtered chain homotopy. It follows from the proof of Theorem \ref{rankinnequalityforfiberedlinks} that $\langle(\partial_T^{-})^2(x),Ux\rangle\neq 0$, a contradiction. Thus $\rank(\widehat{\HFL}(L;(A_1,\dots A_i-1,\dots A_n)))\geq 2$ for all $i$ and the result follows from the fact that knot Floer homology is obtained from link Floer homology by collapsing the multi-Alexander grading to a single Alexander grading.
\end{proof}

For Section~\ref{Knot Floer Homology Detects $T(2,8),T(2,10)$} and Section~\ref{Detection Results for Khovanov Homology} it will in fact be helpful to have the following algebraic generalization of Corollary~\ref{Nigeneralized}:

\begin{proposition}~\label{prop:algebraicni}
    Suppose $L$ is a link such that:\begin{enumerate}
        \item $\rank(\widehat{\HFL}(L,(A_1+1,A_2,\dots A_n)))=0$
        \item $\rank(\widehat{\HFL}(L;(A_1,\dots A_n)))=1$, with $x$ a generator of $\widehat{\HFL}(L;(A_1,\dots A_n))$ and $\bar{x}$ a generator of $\widehat{\HFL}(L;(-A_1,\dots -A_n))$, such that $\partial_T x$ has a non-trivial component in $\widehat{\HFL}(L;(A_1-1,A_2\dots A_n))$
        
        \item there exists a generator $y$ of $\widehat{\HFL}(L;(1-A_1,-A_2,\dots -A_n))$ such that $\langle \partial_T y,\bar{x}\rangle\neq 0$.
    \end{enumerate} 

    Then $\rank(\widehat{\HFL}(L;(A_1-1,A_2\dots, A_n)))>1$.
\end{proposition}

Here $\partial_T$ denotes the total differential on $\widehat{\CFL}(L)$. We have stated the proposition for $A_1$ only for ease of notation, a more general version readily follows by permuting the ordering of the components of $L$. Here recall that $\CFL^-(L)$ can be thought of as a $((A_1,A_2,\dots A_n),(j_1,j_2,\dots j_n))$-graded complex, where the $j_i$ gradings are the $U_i$ gradings, where the action of $U_i$ decreases the $j_i$ and $A_i$ gradings by $1$ and preserves the other gradings. Recall that one can view $\widehat{\CFL}(L)$ as the sub-complex of $\CFL^-(L)$ with $j_i = 0$ for $i=1,2, \cdots n$ or alternately as the sub-complex of $\CFL^-(L)$ with $(A_1,A_2,\dots A_n)$ fixed, if all $A_i$ are sufficiently negative. Let $\partial^-_T$ again denote the total differential on $\CFL^-(L)$, that counts all pseudo-holomorphic disks.

\begin{proof}
    Consider a reduced model of $\CFL^-(L)$ i.e. one such that the grading preserving component of $\partial_T^-$ is trivial. This is permissible because the filtered chain homotopy type of 
 $\CFL^-(L)$ is a link invariant and one can apply a filtered change of basis to ensure that the grading preserving component of $\partial_T^-$ is trivial -- see~\cite[Reduction Lemma]{hedden2018geography} in the knot case and note that the proof carries through to the link case. Observe that there is a generator \begin{align*}U_1x\in \CFL^-(L;(A_1-1,A_2,\dots A_n),(-1,0,\dots 0))\end{align*} 
    
     Since $\rank(\widehat{\HFL}(L;(A_1-1,\dots A_n)))\geq 1$, \begin{align*}{\rank(\CFL^-(L;(A_1-1,\dots A_n),(0,0,\dots ,0)))\geq 1}.\end{align*} Suppose towards a contradiction that \begin{align*}\rank(\widehat{\HFL}(L;(A_1-1,\dots A_n)))=1\end{align*} so that \begin{align*}\CFL^-(L;(A_1-1,\dots A_n),(0,0,\dots ,0))\end{align*} is rank one with a generator $\bar{y}$. Observe that by assumption 2) there is a component of $\partial_T^-$ from $x$ to $\bar{y}$. Note that this property is preserved under filtered chain homotopy.
    
     We claim that there is also a component of $\partial_T^-$ from $\bar{y}$ to $U_1x$ as a byproduct of the assumption (3) in the statement of the Proposition. 

     To verify the claim, note that in any fixed Alexander multi-grading $(k,k,\dots k)$, with $k$ sufficiently negative, $\CFL^-(L)$ is given by a filtered chain homotopy equivalent copy of $\widehat{\CFL}(L)$, with the role of the (positive) $A_i$ gradings in $\widehat{\CFL}(L)$ is interchanged with the roles of the (negative) $j_i$ gradings in $\CFL^-(L;\{a_i=k\text{ for all Alexander gradings }a_i\})$ -- see the proof of~\cite[Proposition 8.1]{HolomorphicdiskslinkinvariantsandthemultivariableAlexanderpolynomial}. Forgetting $\partial_T^-$, the isomorphism can be given at the level of the underlying vector space by mapping a generator in Alexander multi-grading $\CFL^-(L;(x_1,x_1,\dots x_n);(0,0,\dots 0))$ to $\CFL^-(L;(k,k,\dots k);(k-x_1,k-x_2,\dots k-x_n))$ by $z\mapsto\underset{1\leq i\leq n}{\prod}U^{x_i-k}_iz$. Indeed, under this isomorphism, in $\CFL^-(L)$ the line through $((-A_1,-A_2,\dots,-A_n),(0,0,\dots,0))$ in the (positive) $A_1$ direction is identified with the line through $((k,k,\dots k),(k-A_1,k-A_2,\dots k-A_n))$ in the (negative) $j_1$ direction. In particular, $\CFL^-(L;(k,k,\dots k),(k-A_1+1,k-A_2, \dots k-A_n))$ is generated by:
     
     \begin{align*} U_1^{-1}\underset{1\leq i\leq n}{\prod}U_i^{A_i-k}\bar{y} \end{align*}
     
     while $\CFL^-(L;(k,k,\dots k),(k-A_1,k-A_2,\dots k-A_n))$ is generated by
     \begin{align*}
    \underset{1\leq i\leq n}{\prod}U_i^{A_i-k}x.\end{align*}
       
       It follows from condition (3) that there is a component of $\partial^-_T$ from $U_1^{-1}\underset{1\leq i\leq n}{\prod}U_i^{A_i-k}\bar{y} $ to $\underset{1\leq i\leq n}{\prod}U_i^{A_i-k}x$. It follows in turn that there is a component of the differential from $\bar{y}$ to $U_1x$, as desired.
    
    Now, since $\rank(\widehat{\HFL}(L;(A_1+1,A_2,\dots A_n)))=0$, we have that \begin{align*}
        \rank(\CFL^-(L;(A_1,A_2,\dots A_n),(-1,\dots 0)))=0
    \end{align*}

    It follows that $\langle(\partial^-_T)^2x',x\rangle\neq 0$, contradicting $(\partial_T^-)^2=0$.
\end{proof}

	We conclude this section by noting that it is perhaps useful to view Theorem~\ref{rankinnequalityforfiberedlinks} in the context of the following more general proposition:

\begin{proposition}\label{generalrankbound}
Suppose $L$ is an $n>1$ component non-trivial link in a rational homology sphere such that \\ $\rank(\widehat{\HFL}(L,(x_1,x_2,\dots x_n))=2m+1$ for some $m\in \mathbb{N}$. Then $\rank(\widehat{\HFL}(L))\geq 2^{n}+2m$.
\end{proposition}

For instance, this proposition yields non-trivial lower bounds on the ranks of fibered links of two or more components -- they must have rank at least $2^n$. Note that this improves the lower bound $\rank(\widehat{\HFL}(L))\geq 2^{n-1}$ coming from the spectral sequence from $\widehat{\HFL}(L)$ to $V^{n-1}$.

\begin{proof}[Proof of Proposition~\ref{generalrankbound}]
   Suppose $L$ is an $n>1$ component link non-trivial such that $\rank(\widehat{\HFK}(L,(x_1,x_2,\dots x_n)))$ $=2m+1$. For each choice of $i$ there exist an odd number of generators with $A_j$ grading $x_j$ for every $i\neq j$, and $A_i\neq x_j$, since $\widehat{\HFL}(L)$ must have even rank in each such hyperplane. Recursively, we can find additional generators for each choice of hyperplane defined by $A_i=x_i$ for $i$ in any subset of $\{1,2,\dots n\}$. Since each of the prior stages have odd number of generators, we obtain an additional odd numbers of generators at each stage -- again since the rank of $\widehat{\HFL}(L)$ must be even. Thus we obtain at least an additional $2^n-1$ generators.
\end{proof}

Note that if $L$ is an $n$ component link and $\rank(\widehat{\HFK}(L,A))$ is odd for some Alexander grading $A$, then there is a multi-Alexander grading $(A_1,A_2\dots A_n)$ such that $\rank(\widehat{\HFL}(L,A_1,A_2,\dots A_n))$ is odd, where $A=A_1+A_2+\dots A_n$ whence it follows that $\rank(\widehat{\HFK}(L))\geq 2^n$.

\end{section}

\begin{section}{Detection Results for $(2,2n)$-cables}\label{Detection Results for $(2,2n)$-cables}

In this section we provide detection results for the simplest $2$ component cables of the trefoils and figure eight knot. Some care with coefficients is required.

\begin{theorem}\label{HFL2detectsRHTcables}
Link Floer homology with $\Z_2$ coefficients detects $T(2,3)_{2,2n}$ for all $n$.
\end{theorem}

\begin{theorem}\label{HFLdetects(2,2n)cablesofthefigureeight}
Let $K$ be the figure eight knot. Link Floer homology with $\Z_2$ coefficients detects $K_{2,2n}$ for all $n$.
\end{theorem}

Each of these results is proven in three cases: where $n>0$, $n<0$ and $n=0$. Only the proof of the $n=0$ cases requires $\Z_2$ coefficients. Throughout this section we will use $K_{2,2n}$ to indicate the $(2,2n)$-cable of $K$ oriented so that the two components of $K$ are oriented in parallel, and $K_{\widetilde{(2,2n)}}$ to indicate the $(2,2n)$-cable of $K$ oriented so that it bounds an annulus. We take $\F$ to be $\Z,\Z_2$ or $\Q$ unless otherwise stated.

\begin{lemma}\label{(2,0)trefoilcable}
   Suppose $L$ is a link such that $\widehat{\HFK}(L;\Z_2)\cong\widehat{\HFK}(T(2,3)_{\widetilde{(2,0)}};\Z_2)$. Then $L$ is either isotopic to $T(2,3)_{\widetilde{(2,0)}}$ or the disjoint union of an unknot and a knot $K$ with knot Floer homology given by:\footnote{Since this article first appeared, Baldwin-Sivek have shown that the only knot with this knot Floer homology is the mirror of $5_2$~\cite{baldwin2022floer}}\begin{align*}
        \widehat{\HFK}(K)\cong(\Z_2)_0^2[1]\oplus(\Z_2)_{-1}^3[0]\oplus(\Z_2)_{-2}^2[-1].
    \end{align*}
\end{lemma}

 Note that this result implies that link Floer homology with $\Z_2$ coefficient detects $T(2,3)_{(2,0)}$, since if a link $L$ satisfies $\widehat{\HFL}(L)\cong\widehat{\HFL}(T(2,3)_{(2,0)})$, then $L$ does not contain an unlinked, unknotted component. This follows from the fact that for any link $L'$ and unknot $U$, the following is true: 

 \[\widehat{\HFL}(L' \sqcup U) = \widehat{\HFL}(L') \otimes (\mathbb{F}_{0} \oplus \mathbb{F}_{-1})\]

\vspace{.1in}

where here $\mathbb{F}_{0} \oplus \mathbb{F}_{-1}$ is supported in multi-Alexander grading $0$.

    We first give a partial computation of $\widehat{\HFK}(T(2,3)_{\widetilde{(2,0)}};\Z_2)$. 
    
    \begin{lemma}\label{20trefoilcomp}
  Let $\F$ be $\Z_2$. $ \widehat{\HFK}(T(2,3)_{\widetilde{2,0}};\F)$ is given by either:  \begin{align*}
 (\F_{\frac{1}{2}}^2\oplus\F_{-\frac{1}{2}}^2)[1]\oplus(\F_{-\frac{1}{2}}^3\oplus\F_{-\frac{3}{2}}^3)[0]\oplus(\F_{-\frac{3}{2}}^2\oplus\F_{-\frac{5}{2}}^2)[-1]   \end{align*}

or

\begin{align*}(\F_{\frac{1}{2}}^2\oplus\F_{-\frac{1}{2}}^2)[1]\oplus(\F_{-\frac{1}{2}}^4\oplus\F_{-\frac{3}{2}}^4)[0]\oplus(\F_{-\frac{3}{2}}^2\oplus\F_{-\frac{5}{2}}^2)[-1]  \end{align*}
    \end{lemma}

    To prove this we compute $\widehat{\HFK}(T(2,3)_{(2,1)};\Z_2)$ and $\widehat{\HFK}(T(2,3)_{(2,-1)};\Z_2)$ using Hanselman-Watson's cabling formula in the theory of immersed curves~\cite{hanselman2023cabling}, then apply the skein exact triangle for knot Floer homology to determine $\widehat{\HFK}(T(2,3)_{(2,0)};\Z_2)$, 
    and thereby $\widehat{\HFK}(T(2,3)_{\widetilde{(2,0)}};\Z_2)$.    The reason we use $\Z_2$ coefficients is that the immersed curves formulation of bordered Floer homology with integer coefficients and indeed bordered Floer homology with integer coefficients itself have not yet been developed. 
    
\begin{proof}[Proof of Lemma~\ref{20trefoilcomp}] Hanselman--Watson's cabling formula implies that: \begin{align*}
        \widehat{\HFK}(T(2,3)_{(2,1)};\F)\cong \F_0[2]\oplus(\F_{-1}\oplus\F_0)[1]\oplus\F_{-1}[0]\oplus (\F_{-2}\oplus\F_{-3})[-1]\oplus\F_{-4}[-2]
    \end{align*}
    
    \begin{align*}
        \widehat{\HFK}(T(2,3)_{(2,-1)};\F)\cong \F_2[2]\oplus(\F_{1}\oplus\F_0)[1]\oplus(\F_{-1}^2\oplus\F_0)[0]\oplus (\F_{-2}\oplus\F_{-1})[-1]\oplus\F_{-2}[-2]
    \end{align*}
    
    \vspace{.2in}
    
    \noindent We apply the skein exact triangle from~\cite[Equation 7]{Holomorphicdisksandknotinvariants}, taking $L_-$ to be $T(2,3)_{(2,-1)}$, $L_+$ to be $T(2,3)_{(2,1)}$ and $L_0$ to be $T(2,3)_{(2,0)}$. Thus we can deduce that $\widehat{\HFK}(T(2,3)_{(2,0)},i)$ is trivial for $i> 2$ and $\widehat{\HFK}(T(2,3)_{(2,0)},2)\cong \F_{\frac{3}{2}}\oplus\F_{\frac{1}{2}}$. To determine $\widehat{\HFK}(T(2,3)_{(2,0)},1)$ we note that it follows immediately from the skein exact triangle that it is either $\F_{\frac{1}{2}}\oplus\F_{-\frac{1}{2}}$ or $\F_{\frac{1}{2}}^2\oplus\F_{-\frac{1}{2}}^2$. To exclude the former case first note that $\widehat{\HFL}(T(2,3)_{(2,0)})$ is supported on the lines $A_2=A_1$, $A_2=A_1+1$, $A_2=A_1-1$, since $T(2,3)_{\widetilde{(2,0)}}$ bounds an annulus. It follows that the component of $\widehat{\HFL}(T(2,3)_{(2,0)})$ with Alexander gradings satisfying $A_1+A_2=2$ is $(\F_{1}\oplus\F_{0})[1,1]$, where $[1,1]$ indicates the Alexander multi-grading . Since $T(2,3)$ is fibered, but neither component of $T(2,3)_{(2,0)}$ is braided with respect to the other (since the linking number is zero), it follows from \cite{martin2022khovanov} that the rank of $\widehat{\HFL}(T(2,3)_{(2,0)}))$ is at least four in each of the maximal $A_i$ gradings (namely $A_i=1$). Thus $\widehat{\HFK}(T(2,3)_{(2,0)},1)\cong \F_{-\frac{1}{2}}^2\oplus\F_{\frac{1}{2}}^2$ and indeed: \begin{align*}\widehat{\HFL}(T(2,3)_{(2,0)},(1,0))\cong\widehat{\HFL}(T(2,3)_{(2,0)};(0,1))\cong\F_0\oplus\F_{-1}.\end{align*}
    
    To conclude the computation it suffices to determine $\widehat{\HFK}(T(2,3)_{(2,0)},0)$ by the symmetries of link Floer homology. Note again that the skein exact triangle implies that $\widehat{\HFK}(T(2,3)_{(2,0)},0)$ is either $\F_{-\frac{1}{2}}\oplus\F_{-\frac{3}{2}}$ or $ \F_{-\frac{1}{2}}^2\oplus\F_{-\frac{3}{2}}^2$.
    
    It follows that either: \begin{align*}
\widehat{\HFK}(T(2,3)_{\widetilde{2,0}})\cong (\F_{\frac{1}{2}}^2\oplus\F_{-\frac{1}{2}}^2)[1]\oplus(\F_{-\frac{1}{2}}^3\oplus\F_{-\frac{3}{2}}^3)[0]\oplus(\F_{-\frac{3}{2}}^2\oplus\F_{-\frac{5}{2}}^2)[-1]   \end{align*}

or

\begin{align*} \widehat{\HFK}(T(2,3)_{\widetilde{2,0}})\cong (\F_{\frac{1}{2}}^2\oplus\F_{-\frac{1}{2}}^2)[1]\oplus(\F_{-\frac{1}{2}}^4\oplus\F_{-\frac{3}{2}}^4)[0]\oplus(\F_{-\frac{3}{2}}^2\oplus\F_{-\frac{5}{2}}^2)[-1]  \end{align*}

\end{proof}

We can now proceed to the proof of Lemma~\ref{(2,0)trefoilcable}
\begin{proof}[Proof of Lemma~\ref{(2,0)trefoilcable}]
 Suppose $\widehat{\HFK}(L)\cong\widehat{\HFK}(T(2,3)_{\widetilde{(2,0)}})$. It suffices to show that $L$ bounds an annulus, has linking number $0$ and has  a trefoil component.

We first note that $L$ has at most two components since the maximal Maslov grading of $\widehat{\HFK}(L)$ is $\frac{1}{2}$. Since $\rank(\widehat{\HFK}(L))$ is even it follows that $L$ cannot be a knot, so $L$ is a two component link. The fact that the linking number is zero follows from the fact that the Conway polynomial detects the linking number of two component links~\cite{hoste1985firstcoefficientoftheconwaypolynomial}. Since the maximal Alexander grading of $\widehat{\HFK}(L)$ is $1$, $L$ either bounds an annulus or it is the disjoint union of an unknot and a genus one knot. If $\rank(\widehat{\HFK}(L))=16$ then $L$ cannot be a two component link with a split, unknotted component as in this case the genus one knot would have to have knot Floer homology of even rank, a contradiction. Thus the genus one knot must have knot Floer homology of the form: \begin{align*}
    \F_0^2[1]\oplus\F_{-1}^3[0]\oplus\F_{-2}^2[-1]
\end{align*} as required.

Suppose now that $L$ bounds an annulus. Note that a two component link containing a knot of genus $g$ must have knot Floer homology containing generators of Maslov grading which differ by at least $1+2g$. Since the Maslov gradings of $\widehat{\HFK}(L)$ differ by at most $3$ it follows that each component of $L$ is of genus at most one. Note that the only link with unknotted components bounding an annulus and of linking number zero is the two component unlink, which is of the incorrect Euler characteristic. Thus $L$ contains genus one components. To see that each of these components is fibered, we note that if $L$ has a component $K$ with $g(K)=1$ which is not fibered then $L$ contains at least two pairs of generators which differ in Maslov grading by $3$. There are only two such pairs, namely the pairs of generators of Maslov grading $\frac{1}{2}$ and $-\frac{5}{2}$. Since the linking number of $L$ is zero we see that in link Floer homology the two Maslov index $0$ generators have, without loss of generality, $A_1$ grading equal to $1$, so we have that they are supported in bi-Alexander grading $(1,0)$. But since $L$ bounds an annulus, both components of $L$ are isotopic and we must have two generators of $A_2$ grading $1$, a contradiction. Thus $L$ has genus one fibered components.

It thus suffices to show these components are neither $T(2,-3)$ nor the figure eight. Note that $T(2,-3)_{\widetilde{(2,0)}}$ is the mirror of $T(2,3)_{(2,0)}$, therefore  $$\widehat{\HFK}(T(2,-3)_{\widetilde{(2,0)}})\cong(\widehat{\HFK}(T(2,3)_{\widetilde{(2,0)}}))^*\not\cong\widehat{\HFK}(L).$$ Similarly, the $\widetilde{(2,0)}$ cable of the figure eight is isotopic to the $\widetilde{(2,0)}$-cable of the mirror of the figure eight, so if $L$ where the $\widetilde{(2,0)}$-cable of the figure eight then its knot Floer homology would be isomorphic to its dual, which is not the case. Thus $L$ is isotopic to $T(2,3)_{\widetilde{(2,0)}}$ as desired.

\end{proof}

We now prove a similar result for the $(2,0)$ cable of the figure eight knot.

\begin{lemma}\label{20cablefigureeight}
    Let $K$ be the figure eight knot. Suppose $\widehat{\HFL}(L)\cong\widehat\HFL(K_{2,0},\Z_2)$. Then $L$ is isotopic to $K_{2,0}$.
\end{lemma}

As before we begin with a partial computation of $\widehat\HFK(K_{2,0})$.

\begin{lemma}
    Let $\F$ be $\Z_2$ and $K$ be the figure eight knot. $\widehat{\HFK}(K_{{(2,0)}},i;\F)$ is $0$ for $i> 2$, $\F_{\frac{3}{2}}\oplus\F_{\frac{5}{2}}$ for $i=2$, $\F_{\frac{3}{2}}^2\oplus\F_{\frac{1}{2}}^2$ for $i=1$, and $(\F_{\frac{1}{2}}\oplus\F_{-\frac{1}{2}})^n$ for $i=0$ and some $n\leq 3$.
\end{lemma}
\begin{proof}
Hanselman-Watson's cabling formula for knot Floer homology in terms of immersed curves~\cite{hanselman2023cabling} shows that: \begin{align*}
\widehat{\HFK}(K_{2,1})\cong \widehat{\HFK}(K_{2,-1})\cong \F_1[2]\oplus(\F_1\oplus\F_0)[1]\oplus\F_0^3[0]\oplus(\F_{-2}\oplus\F_{-1})[-1]\oplus\F_{-3}[-2]\end{align*}

Again we apply the skein exact triangle from~\cite[Equation 7]{Holomorphicdisksandknotinvariants} -- with $L_0=K_{2,0}, L_+=K_{2,1}, L_-=K_{2,1}$. This immediately implies that $\widehat{\HFK}(K_{2,0},i)\cong 0$ for $|i|>2$ and $\widehat{\HFK}(K_{2,0},2) \cong 
\F_{\frac{3}{2}}\oplus\F_{\frac{5}{2}}$. Indeed we see that $\widehat{\HFK}(K_{2,0},1)$ is either $\F_{\frac{3}{2}}^2\oplus\F_{\frac{1}{2}}^2$ or $\F_{\frac{3}{2}}\oplus\F_{\frac{1}{2}}$. To rule out the latter case, we apply similar tactics as in after stating Lemma \ref{(2,0)trefoilcable}. In brief, note that $K_{\widetilde{2,0}}$ bounds an annulus, so that $\widehat{\HFL}(K_{2,0})$ is supported on the lines $A_2=A_1$, $A_2=A_1\pm 1$. Thus since $\widehat{\HFL}(K_{2,0})$ must be of even rank in each $A_i$ grading, and in the gradings $A_i=1$ the rank of the link Floer homology of $L$ must be at least $4$ since $L$ is not fibered, the case that $\widehat{\HFK}(K_{2,0},1)\cong \F_{\frac{3}{2}}\oplus\F_{\frac{1}{2}}$ is excluded. While unnecessary we note that $\widehat{\HFK}(K_{2,0},0)\cong(\F_{\frac{1}{2}}\oplus\F_{-\frac{1}{2}})^n$ for some $0\leq n\leq 3$. In fact $n\geq 2$, as else $\widehat{\HFK}(K_{2,0})$ cannot admit a spectral sequence to $\widehat{\HFL}(K)\otimes V$.
\end{proof}

\begin{proof}[Proof of Lemma~\ref{20cablefigureeight}]
Suppose $L$ is as in the statement of the Lemma. Then $L$ is a two component link with linking number $0$. Note that upon reversing the orientation of a component, $L$ bounds a surface of Euler characteristic $0$. Thus either $L$ consists of the disjoint union of an unknot and a genus one knot, or $L$ bounds an annulus. The former case is excluded since $\widehat{\HFL}(L)$ is not supported along a single $A_i=0$ grading. Thus $L$ bounds an annulus, and is therefore the $(2,0)$-cable of some knot $K$.

To see that $K$ is the figure eight knot, note that $K$ can be of genus at most one since the maximal $A_i$ grading is $1$ for each $i$. Indeed, it is plainly not the case that $\widehat{\HFL}(K_{2,0})\cong\widehat{\HFL}(T(2,0))$, so $K$ must be a genus one knot. To see that $K$ is fibered, note that the Maslov gradings of the $A_1=1$ part of $\widehat{\HFL}(L)$ is given by $\F_0\oplus\F_2\oplus\F_1^2$ while the $A_1=-1$ part is given by $\F_{-1}\oplus\F_{-2}^2\oplus\F_{-3}$. It follows that the Maslov index $2$ generator, as well as the Maslov index $-3$ generator cannot persist under the spectral sequence to $\widehat{\HFL}(L_1)$, where $L_1$ is the component of $L$ corresponding to $A_1$. Thus $\widehat{\HFL}(L_1)$ is of rank one in the top Alexander grading and $L_1$ is fibered~\cite{ghiggini2008knot}. Thus $L$ is a $(2,0)$-cable of a trefoil or the figure eight. But the $(2,0)$-cables of the trefoils have distinct knot Floer homology, as shown previously. Thus $L$ is $K_{2,0}$ as desired.
\end{proof}

To prove Theorem~\ref{HFL2detectsRHTcables} and Theorem~\ref{HFLdetects(2,2n)cablesofthefigureeight} we first show that if the link Floer homology of a link satisfies certain properties then the link is a $(2,2n)$ cable of a trefoil or the figure eight knot. Then we show that the link Floer homology of each of these three links satisfies these properties for the given $n$.

\begin{lemma}\label{algebraiccharachterizationofcables}
Let $L$ be a two component link with linking number $n>1$ such that:
\begin{enumerate}
        \item $\widehat{\HFL}(L)$ is of rank $2$ in the maximal non-zero $A_1$ or $A_2$ grading.
        \item The link Floer homology polytope has a vertex at multi-Alexander grading $(1+\frac{n}{2},1+\frac{n}{2})$.
        \item $\widehat{\HFL}(L)$ has support on the lines $A_2=A_1,A_2=A_1+1, A_2=A_1-1$.
         \end{enumerate}
         
Then $L$ is isotopic to $K_{2,2n}$ where $K$ is a trefoil or the figure eight knot.

\end{lemma}

\begin{proof}[Proof of Lemma~\ref{algebraiccharachterizationofcables}] Suppose $L$ is as in the hypotheses of the theorem. Condition $3$ implies that $L$ bounds an annulus upon reversing the orientation of either component -- it bounds an Euler characteristic zero surface with two boundary components, neither of which is a disk. Since the maximal Alexander grading is $2+n$, it follows that the maximal $A_i$ gradings are $1+\frac{n}{2}$.

Since  $g(L_i)+\frac{n}{2}\leq\max\{A_i:\widehat{\HFL}(L,(A_1,A_2))\neq 0\}$, it follows that $g(L_i)\leq 1$ for each $i$. Indeed since, $\widehat{\HFL}(L)$ is of rank $2$ in $A_1$ grading $\frac{n}{2}+1$, $L_i$ is genus one and fibered or genus zero, and hence an unknot, a trefoil or the figure eight knot. It follows that $L\cong K_{2,2n}$ where $K$ is the unknot, a trefoil, or the figure eight knot. $(2,2n)$-cables of the unknot are the torus links $T(2,2n)$ which do not satisfy property $2$. Thus $L$ is a $(2,2n)$ cable of a trefoil knot or the figure eight knot, as desired.
\end{proof}
 
 To apply Lemma~\ref{algebraiccharachterizationofcables}, we need to show that $\widehat{\HFL}(K_{2,2n})$ satisfies the conditions of the hypothesis of Lemma~\ref{algebraiccharachterizationofcables} for $K$ a trefoil or figure eight knot. The easiest way to proceed is via the following remark:
 
 \begin{remark}
Rudolph shows that $K_{n,m}$ can be written as a Murasugi sum of $K_{n,\text{sign}(m)}$, which is fibered by a result of Stallings~\cite{stallings1978constructions}, and $T(n,m)$ along minimal genus Seifert surfaces~\cite{rudolph2002non}. Note that Rudolph's result is only stated for $n,m$ coprime, but the proof works for arbitrary $n,m$. Results of Gabai show that Murasugi sums of fiber surfaces are fibered, and that Murasugi sums of minimal genus Seifert surfaces are minimal genus Seifert surfaces~\cite{gabai1986detecting}. It follows that $(2,2n)$ cables of genus one fibered knots are fibered and of genus $n$, that is Euler characteristic $-2n$. \end{remark}

 If one wishes to only use Rudolph's stated result, for knots, one can alternately proceed via the following lemma:
 
 \begin{lemma}\label{genusof22ncables}
 Let $K$ be a genus one fibered knot. For $n\geq 1$, $K_{2,2n}$ is fibered and $\chi(K_{2,2n})=-2n$.
 \end{lemma}
 
 \begin{proof}[Proof of Lemma~\ref{genusof22ncables}]
 Note that, for $n>0$, $K_{2,2n+1}$ is the Murasugi sum of $K_{2,1}$ and $T(2,2n+1)$~\cite{neumann1987unfoldings, rudolph2002non} along their fiber surfaces, which are minimal genus. Let $g$ be a genus one fibered knot. It follows from a result of Gabai~\cite{gabai1986detecting} that $K_{2,2n+1}$ is fibered and that $g(K_{2,2n+1})=2+n$. From here, using the Skein exact triangle for knot Floer homology, we see that $\dfrac{2-\chi(K_{2,2n})}{2}=1+n$ for all $n\geq 1$ -- that is $\chi(K_{2,2n})=-2n$ -- and that ${\rank(\widehat{\HFK}(K_{2,2n},1+n))=1}$ so that $K_{2,2n}$ is fibered, as desired.
 \end{proof}

It follows that conditions 1, 2 and 3 of Lemma~\ref{algebraiccharachterizationofcables} hold for $(2,2n)$ cables of the trefoils, and figure eight. To see that condition $3$ holds note that $(2,2n)$ cables of the three links bound an annuli under reversing the orientation of one of the components.

 We can now proceed to the detection results.

\begin{proof}[Proof of Theorem~\ref{HFL2detectsRHTcables}]

Suppose $L$ is as in the statement of the Theorem. We proceed by cases: when $n=0$, $n>0$ and when $n<0$. The $n=0$ case was dealt with in Lemma~\ref{(2,0)trefoilcable}.

Suppose now that $n>0$. We first show that $\widehat{\HFL}(T(2,3)_{2,2n})$ satisfies the conditions in the hypothesis of Lemma~\ref{algebraiccharachterizationofcables}. Observe that for $n\geq 1$, each component of $T(2,3)_{2,2n}$ is a braid in the complement of the other, so $\widehat{\HFL}(T(2,3)_{2,2n}))$ must be rank two in the maximal $A_i$ grading, by braid detection ~\cite[Proposition 1]{martin2022khovanov}, so condition 1 of Lemma~\ref{algebraiccharachterizationofcables} holds. By Lemma~\ref{genusof22ncables} the maximal Alexander grading of $\widehat{\HFK}(L)$ is $2+n$. $T(2,-3)_{2,2n+1}$ is fibered by Lemma~\ref{genusof22ncables}. Therefore, since $\widehat{\HFL}(T(2,-3)_{2,2n})$ is rank one in the maximal Alexander grading, and must be symmetric under interchanging the components, it follows that the link Floer homology polytope of $L$ has a vertex at $(1+\frac{n}{2},1+\frac{n}{2})$, so that condition 2 holds. We note that the two generators of the top $A_i$ grading of $\widehat{\HFL}(T(2,-3)_{2,2n})$ must be of Maslov gradings $0$ and $-1$.

Suppose now that $\widehat{\HFL}(L)\cong\widehat{\HFL}(T(2,3)_{2,2n})$ for some $n\geq 1$. Since link Floer homology detects the linking number of two component links~\cite{hoste1985firstcoefficientoftheconwaypolynomial}, it follows that the components of $L$ have linking number $n$. Lemma~\ref{algebraiccharachterizationofcables} implies that $L$ is the $(2,2n)$ cable of a trefoil or the figure eight. To see that $L$ is the cable of $T(2,3)$, observe that the $\widehat{\HFL}(L)$ has a Maslov index $0$ generator with $A_1$ grading $1+\frac{n}{2}$ which persists under the spectral sequence to $\widehat{\HFL}(L_i)$, whence $L_i$ is $T(2,3)$.

We now proceed to the case $n<0$, i.e. we show that Link Floer homology detects $T(2,-3)_{2,2n}$ for $n>0$.

Consider $\widehat{\HFL}(T(2,-3)_{2,2n})$. Observe that for $n\geq 1$, each component of $L$ is a braid in the complement of the other, so $\widehat{\HFL}(L)$ must be rank two in the maximal $A_i$ grading, by braid detection~\cite{martin2022khovanov}, so Condition 1 of Lemma~\ref{algebraiccharachterizationofcables} holds. By Lemma~\ref{genusof22ncables} the maximal Alexander grading of $\widehat{\HFK}(L)$ is $2+n$. Since $T(2,-3)_{2,2n+1}$ is fibered, and $\widehat{\HFL}(T(2,-3)_{2,2n+1})$ must be symmetric under interchanging the components, it follows that the link Floer homology polytope of $L$ has a vertex at $(1+\frac{n}{2},1+\frac{n}{2})$, so that condition 2 of Lemma \ref{alexandergradingcharachterizationofunlinks} holds. Condition 3 holds since $T(2,-3)_{2,2n}$ bounds an annulus upon reversing the orientation of a component. Note that the two generators of the top $A_i$ grading of $\widehat{\HFL}(T(2,-3)_{2,2n})$ must be of Maslov gradings $1$ and $2$.

Suppose now that $\widehat{\HFL}(L)\cong\widehat{\HFL}(T(2,-3)_{2,2n})$. Lemma~\ref{algebraiccharachterizationofcables} implies that $L$ is a $(2,2n)$ cable of a trefoil or the figure eight knot. Since the two generators of the top $A_i$ grading of $\widehat{\HFL}(L)$ must be of Maslov gradings $1$ and $2$, it follows that each component is $T(2,-3)$ and we have the desired result.
\end{proof}

The proof of Theorem~\ref{HFLdetects(2,2n)cablesofthefigureeight} is again similar.

\begin{proof}[Proof of Theorem~\ref{HFLdetects(2,2n)cablesofthefigureeight}]
The proof follows the proof of Theorem~\ref{HFL2detectsRHTcables} verbatim, aside from to note that the top $A_i$ generators must be of Maslov index $0$ and $1$ respectively for $n\geq 0$.
\end{proof}

\end{section}

\begin{section}{$E_2$ collapsed chain complexes}
\label{Technical Results}

	In this section we prepare the ground for our detection results for $T(2,8)$ and $T(2,10)$ in subsequent sections by proving a purely algebraic result for $E_2$ collapsed chain complexes.

We will rely on a structure theorem for \emph{$E_2$-collapsed chain complexes over $\Q$} i.e. $\Z\oplus\Z$ graded finitely generated chain complexes with coefficients in $\Q$ and differentials of length at most one. The $\Z_2$ coefficient version is provided in \cite[Section 12]{HolomorphicdiskslinkinvariantsandthemultivariableAlexanderpolynomial}, see \cite[Lemma 7]{petkova2013cables} for a similar result. In particular Ozsv\'ath-Szab\'o note that each $E_2$-collapsed spectral sequences over $\Z_2$ split as the direct sum of complexes of one of five forms, $B_d, V_d^l,H_d^l,X_d^l$ and $Y_d^l$. These are shown in the Figure \ref{e2col}. For the $\Q$ coefficient version the only necessary modification is that the arrows correspond to maps $\mathbbm{1}:\Q\to\Q$, save for the top arrow of the $B_d$ complex, which is given by $-\mathbbm{1}$.

\begin{figure}[h!]
    \centering
    \includegraphics[width=7cm]{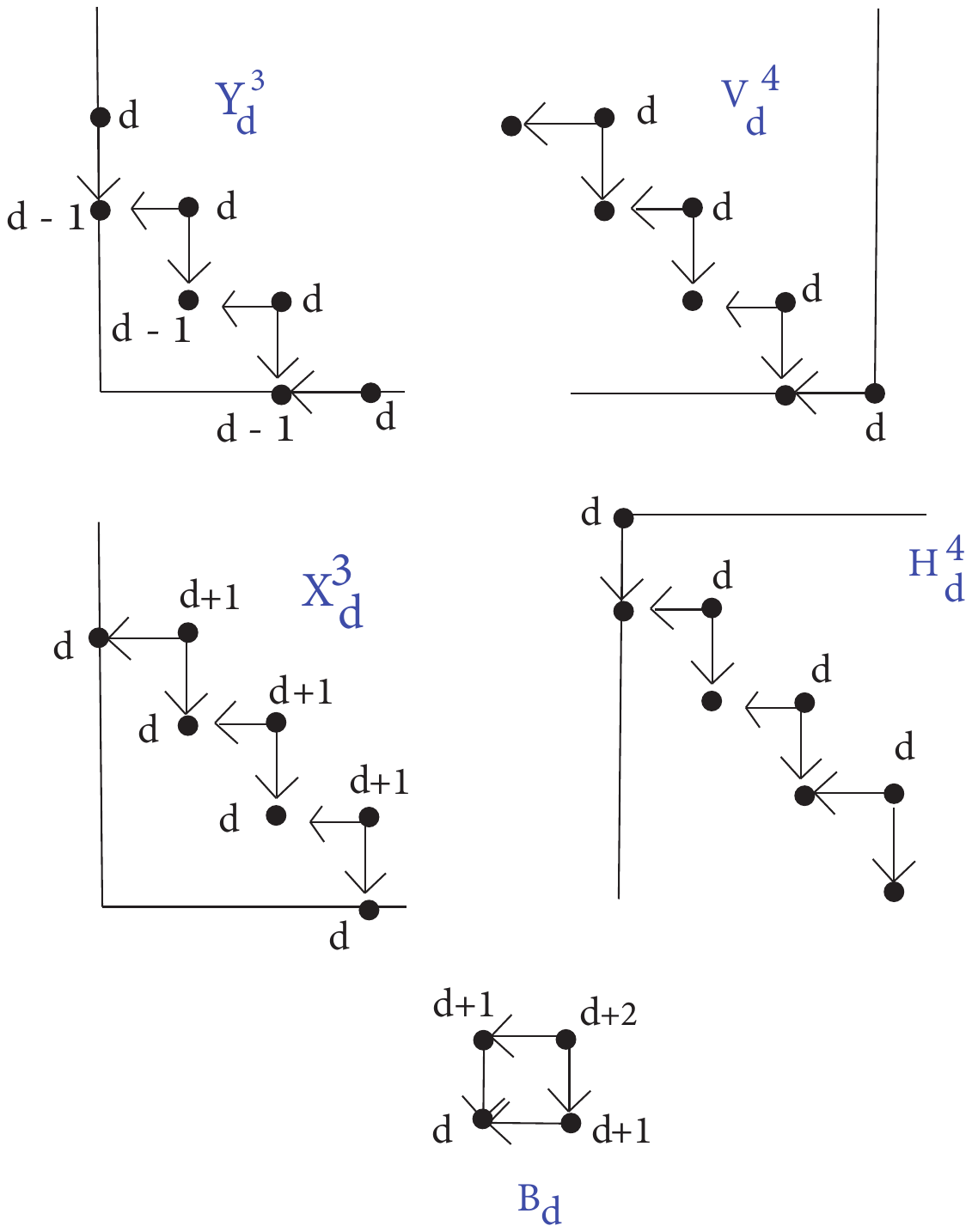}
    \caption{Examples of possible summands of an $E_2$-collapsed chain complex.}
    \label{e2col}
    \end{figure}

    \begin{proposition}\label{E2collapsed}
    Suppose $(C,\partial)$ is a bi-filtered $E_2$ collapsed chain complex with coefficients in $\Q$. Then $(C,\partial)$ is bi-filtered chain homotopy equivalent to a complex which decomposes as a direct sum of copies of chain complexes of the type $B_d[i,j]\oplus V_d^l[i,j],H_d^l[i,j],X_d^l[i,j], Y_d^l[i,j]$.
    \end{proposition}
    
    Here and henceforth, $[i,j]$ indicates a shift up in $A_1$ grading by $i$ and a shift up in $A_2$ grading by $j$. Since for a two component $\delta$-thin link $L$ $\widehat{\CFL}(L)$ is $E_2$ collapsed, this result is of great utility. Indeed, the filtered chain homotopy types of the components of $L$ yield significant restraints on the filtered chain homotopy type of $\widehat{\CFL}(L)$. For instance, since the full filtered $\widehat{\CFL}$ complex has homology $\F_0\oplus\F_{-1}$, we see that $\widehat{\CFL}(L)$ has a summand $X_0^l$ or $Y_0^l$ and a summand $X_{-1}^l$ or $Y_{-1}^l$, and no other summands of the form $X_d^l,Y_d^l$. Similarly, it is readily seen that the chain homotopy type of $\widehat{\CFL}(L_1)$ determines the $V_d^l$ summands, while $\widehat{\CFL}(L_2)$ determines the $H_d^l$ summands. Specifically, if there are generators $x,y\in\widehat{\CFL}(L_1)$, with $\langle\partial_{\widehat{\CFL}(L_1)} x,y\rangle\neq 0$ then we must have a summands of $\widehat{\CFL}(L)$ of the form $V_{M(x)}^{A(x)-A(y)}$, as no other summand induces the correct differential on $\widehat{\CFL}(L_1)$.

    The proof of Proposition~\ref{E2collapsed} is entirely elementary. \begin{proof}[Proof of Proposition~\ref{E2collapsed}]
    
    We prove this by showing that $(C,\partial)$ admits a graded basis $\{v_i\}$ for $C$ such that $ \partial v_j$
    is either trivial, a single generator in $\{v_i\}\cap C(x-1,y)\oplus C(x,y-1)$, or the sum of a generator in $\{v_i\}\cap C(x,y-1)$ and a generator in $\{v_i\}\cap C(x-1,y)$. For the purposes of this proof we call a chain complex with a choice of such a basis~\emph{simplified}. The result then follows.

    By a bi-filtered version of the reduction lemma~\cite[Lemma 4.1]{hedden2013khovanov} we can assume that the grading preserving part of $\partial$ is trivial.
    
    Indeed, by repeatedly changing basis we may assume that each summand $C(x,y)$ splits as a direct sum $A(x,y)\oplus B(x,y)\oplus D(x,y)$, where $\partial_2|_{A(x,y)}:A(x,y)\to B(x,y-1)$ is represented by the identity matrix under a choice of basis and $\partial_2$ is trivial when restricted to $B(x,y)\oplus D(x,y)$.
    
    We now proceed to produce a basis for which $\bigoplus_{x,y}C(x,y)$ is simplified. Starting with the basis as above, we proceed by a double induction: we assume the subcomplex $\bigoplus_{(x,y): x< n} C(x,y)$ is simplified and then in turn assume that each summand $\bigoplus_{x<n,y=m} C(x,y)$ is simplified then proceeding by induction on $n$ then $m$.

    Consider $z$ a basis element in  $C(n,m)$. Suppose $\partial_1(z)=\sum q_ia_i+\sum r_jb_j+\sum s_kd_k$, where $a_i$ are basis elements for $A(n-1,m)$, $b_j$ are basis elements for $B(n-1,m)$ and $d_k$ are basis elements for $D(n-1,m)$ and $q_i,r_j,s_k\in\Q$.

    Suppose $r_j=0$ for all $j$. Then we can change basis for $A(n-1,m)\oplus D(n-1,m)\oplus B(n-1,m-1)$ so that $\partial_1 z=q$ for $q$ a basis element in $A(n-1,m)$, as desired.
    
     Likewise if $q_i=0$ for all $i$. We can change basis for $B(n-1,m)\oplus D(n-1,m) \oplus A(n-1,m+1)$ so that $\partial_1 z=q$ for $q$ a basis element in $B(n-1,m)$, as desired.
     
     In general, we may assume after making a change of basis for $A(n-1,m)\oplus D(n-1,m)\oplus B(n-1,m-1)$, and subsequently $B(n-1,m)\oplus D(n-1,m) \oplus A(n-1,m+1)$ that $\partial z=-\epsilon_1a+\epsilon_2b$ with $a$ a basis vector in $A(n,m)$, $b$ a basis vector in $B(n,m)$, $\epsilon_1,\epsilon_2\in\{0,1\}$. If either $\epsilon_1$ or $\epsilon_2$ is non-zero then we are done. Suppose they are both $1$. Note that in order that $\partial_1^2\neq 0$ we must have that $z$ is not in the image of $\partial_1$. Moreover, In order that $\partial^2=0$, we must have that $z\in A(n,m)$. Indeed by induction we have that $\partial_1(a)\neq\partial_1(b)$ so that in fact $\partial_1(a)=\partial_1(b)=0$ in order that $\partial^2=0$. Consider the basis in which we exchange $\{a+b,b\}$ for $\{a,b\}$. Applying this argument for the remaining basis elements of $C(n,m)$ completes the proof.

    \end{proof}

\end{section}

\begin{section}{Knot Floer Homology Detects $T(2,8),T(2,10)$}
\label{Knot Floer Homology Detects $T(2,8),T(2,10)$}

In this section we prove the following detection results:

\begin{theorem}
\label{HFKT28}
	Knot Floer homology detects $T(2,8)$.
\end{theorem}

	\begin{theorem}\label{HFKT210}
	Knot Floer homology detects $T(2,10)$.
\end{theorem}

From here the following corollary follows immediately from work in~\cite{binns2020knot}.

\begin{corollary}
Annular Khovanov homology detects the closures of the braids $\sigma_1\sigma_2\dots\sigma_7$ and $\sigma_1\sigma_2\dots\sigma_{9}$, where $\sigma_i$ are the standard Artin generators of the braid group.
\end{corollary}

Here Annular Khovanov homology is a combinatorial invariant of links in a thickened annulus. For details we refer the reader to~\cite{binns2020knot}.
 
For the readers convenience first we recall that:
    $$\widehat{\HFK}(T(2,2n),i)\cong\begin{cases}\F_{\frac{1+2i-2n}{2}}&\text{ for } i=\pm n\\
    \F^2_{\frac{1+2i-2n}{2}}&\text{ for } -n<i<n\\
    0&\text{ otherwise}
    
    \end{cases}$$
    
    Here $\F$ can be $\Q,\Z$ or $\Z_2$.

We start by showing that in general if $\widehat{\HFK}(L)\cong\widehat{\HFK}(T(2,2n))$ then $L$ consists of two unknotted components and $\widehat{\CFL}(L)$ is of an especially simple form.
\begin{lemma}\label{T22nunknottedcomponents}
    Suppose $\widehat{\HFK}(L;\F)\cong\widehat{\HFK}(T(2,2n);\F)$. Then $(\widehat{\CFL}(L),\partial)$ consists of  $B_d$ summands and a single $Y_0^0[\frac{n}{2},\frac{n}{2}]\oplus Y_{-1}^1[\frac{n-1}{2},\frac{n-1}{2}]$ summand, and $L$ has unknotted components.
\end{lemma}

Here $\F$ is either $\Q$ or $\Z_2$.

\begin{proof}[Proof of Lemma~\ref{T22nunknottedcomponents}]
Let $L$ be as in the Lemma, and $m$ be the number of components of $L$. We first show that $m=2$. To see this note that the maximal Maslov grading of a generator of $\widehat{\HFK}(L)$ is $\frac{1}{2}$, so since $\widehat{\HFK}(L)$ admits a spectral sequence to $\widehat{\HF}(\#^{m-1}S^1\times S^2)$, $L$ can have at most two components. $L$ then has exactly $2$ components since the Maslov grading of $\widehat{\HFK}(L)$ is $\Z+\frac{m-1}{2}$ valued.

Let $L_1$ and $L_2$ be the two components of $L$. It follows that there is a unique Maslov grading $0$ generator. Consider $\widehat{\HFL}(L)$. Since the linking number is $n$ and the Alexander grading of the Maslov index $0$ generator is $n$, it follows that there is a summand $Y_0^0[\frac{n}{2},\frac{n}{2}]$ and since there are two Maslov index $-1$ generators there is a summand $Y_{-1}^1[\frac{n-1}{2},\frac{n-1}{2}]$. As observed in the proof of~\cite[Theorem 12.1]{HolomorphicdiskslinkinvariantsandthemultivariableAlexanderpolynomial}, $\widehat{\HFL}(L)$ consists of pairs of summands $V^l_{d}(x,y)\oplus V^l_{d-1}(x,y-1)$, $H^l_{d}(x,y)\oplus H^l_{d-1}(x,y-1)$, $Y_0^l(x,y)\oplus Y_{-1}^{l-1}(x+1,y+1)$, $X_0^l(x,y)\oplus X^{l-1}_{-1}(x,y)$, as well as a collection of copies of $B_d$.
Note that for the case at hand the rank in each Alexander grading is at most two, so we have that the rest of the complex is given by summands of the form $B_d$ or $V_d^1[x,y]\oplus V_{d-1}^1[x,y-1]$ or $H_d^1[x,y]\oplus H_{d-1}^1[x-1,y]$. Indeed we readily see that $d$ is even.

Suppose $L_1$ is not unknotted. Consider the $V_d^1[x,y]\oplus V_{d-1}^1[x,y-1]$ summands with maximal $x$. Note that $x>\frac{n}{2}$. Amongst these consider the summands with minimal $d$. It follows that there is a generator in $\widehat{\HFK}(L_1)$ of minimal Alexander grading with Maslov grading $d-2(x-\frac{n}{2})$. Note that this is an even number. But this is a contradiction, since if there is a generator of even Maslov index in $A_1$ grading $\frac{n}{2}-x$ which persists under the spectral sequence to $\widehat{\HFK}(L_1)\otimes V$ then there must be a generator of odd Maslov index in $A_1$ grading $\frac{n}{2}-x-1$, since the summands of $\widehat{\HFL}(L)$ that persist under the spectral sequence to $\widehat{\HFK}(L_1)\otimes V$ are of the form $V_d^1[x,y]\oplus V_{d-1}^1[x,y-1]$ with $d$ even. 

Thus there are no $V_d^1[x,y]\oplus V_{d-1}^1[x,y-1]$ summands. A similar proof shows that there are no $H_d^1[x,y]\oplus H_{d-1}^1[x-1,y]$ summands. Thus the remaining summands are all of the form $B_d$ and each component of $L$ is unknotted, as desired.
\end{proof}

We can now prove the two Knot Floer homology detection results.

\begin{proof}[Proof of Theorem~\ref{HFKT28}]
Suppose $\widehat{\HFK}(L)\cong\widehat{\HFK}(T(2,8))$. Lemma~\ref{T22nunknottedcomponents} implies that $L$ is a two component link with unknotted components, and indeed determines the complex up to a choice of $x$ in a summand $B_{-4}[y,-y]$, by the symmetry of link Floer homology. Also since the link Floer homology is of rank $2$ in at least one of the maximal $A_i$ gradings, by \cite[Proposition 1]{martin2022khovanov} it follows that at least one component is a braid axis, and therefore that the maximal $A_i$ grading is exactly $2$.  It follows that $x\in\{1,0,-1\}$. The $y=1$ case can be obstructed using Proposition~\ref{prop:algebraicni}. Specifically observe that $y$ can be taken to be the generator of Alexander multi-grading $(2,-1)$. The $y=-1$ can be obstructed similarly. The result then follows from the fact that Link Floer homology detects $T(2,8)$~\cite[Theorem 3.2]{binns2020knot}.

\end{proof}

The proof of Theorem~\ref{HFKT210} is virtually identical:

\begin{proof}[Proof of Theorem~\ref{HFKT210}]

Suppose $\widehat{\HFK}(L)\cong\widehat{\HFK}(T(2,10))$. Lemma~\ref{T22nunknottedcomponents} implies that $L$ is a two component link with unknotted components, and indeed determines the complex up to a choice of $y$ in a summand $B_{-4}[y+\frac{1}{2},\frac{1}{2}-y]$, by the symmetry of link Floer homology. Since the link Floer homology is of rank $2$ in at least one of the maximal $A_i$ gradings, it follows that at least one component is a braid axis, and therefore that the maximal $A_i$ grading is exactly $2$. It follows that $y\in\{1,0,-1\}$. 
The $y=1$ case can be obstructed using Proposition~\ref{prop:algebraicni}. Specifically observe that $y$ can be taken to be the generator of Alexander multi-grading $(\frac{5}{2},-\frac{1}{2})$. The $y=-1$ can be obstructed similarly. The result then follows from the fact that Link Floer homology detects $T(2,10)$~\cite[Theorem 3.2]{binns2020knot}.
\end{proof}

\end{section}

\begin{section}{Rank Detection Results for Knot Floer Homology }
\label{Rank Detection Results for Knot Floer homology}
	
In this section we show that there are only a small number of links with knot Floer homology of low rank. In particular we show  the following sequence of theorems:
\begin{theorem}\label{rank4}
	If $\rank(\widehat{\HFK}(L))=4$ then $L$ is the Hopf link or the three component unlink.
\end{theorem}

\begin{theorem}\label{rank6}
	Suppose $L$ is a link with $\rank(\widehat{\HFL}(L))=6$. Then $L$ is a disjoint union of a trefoil and an unknot.
\end{theorem}

\begin{theorem}\label{rank8}
The only links $L$ with $\rank(\widehat{\HFL}(L))=8$ are $T(2,4)$, $T(2,-4)$, the four component unlink, and the disjoint union of a Hopf link and an unknot.
\end{theorem}

Indeed, this yields a complete list of links with at least two components of rank at most eight, in combination with the following lemma:

\begin{lemma}\label{unlinkdetection}
Suppose $L$ is an $n+1$ component link with $\rank(\widehat{\HFK}(L))= 2^n$. Then $L$ is an unlink.
\end{lemma}

We can also provide a partial classification of links with knot Floer homology of rank ten:

\begin{proposition}\label{3componentrank12}
There is no three component link with $\rank(\widehat{\HFK}(L))=10$.
\end{proposition}

By similar methods we can show a detection result for knot Floer homology:

\begin{theorem}\label{HFKdetcetstrefoilmeridian}
Let $L_0$ be $T(2,3)$ with a meridian. Suppose $\widehat{\HFK}(L)\cong\widehat{\HFK}(L_0)$. Then $L$ is isotopic to $L_0$.
\end{theorem}
     Here we take the linking number of the right handed trefoil and its meridian to be one. Again this result should be compared to the corresponding link Floer homology detection result in~\cite{binns2020knot}.

The proofs of Theorems \ref{rank4}, \ref{rank6}, \ref{rank8} and \ref{3componentrank12} each essentially follow from case analysis in which we use algebraic and geometric properties of link Floer homology to exclude the unwanted cases, as well as the detection results for Knot Floer homology found in \cite{binns2020knot}. Throughout this section, we take $\F$ to be $\Q$ or $\Z_2$.

We first prove Lemma~\ref{unlinkdetection}, as it is useful in proving the other theorems.

We will find the following characterization of unlinks useful:

\begin{lemma}[Ni~\cite{ni2006note}]\label{alexandergradingcharachterizationofunlinks}
Suppose $\widehat{\HFK}(L)$ is supported in a single Alexander grading. Then $L$ is an unlink.
\end{lemma}

\begin{proof}[Proof of Lemma~\ref{alexandergradingcharachterizationofunlinks}]
Observe that if $\widehat{\HFK}(L)$ is supported in a single Alexander grading then that grading must be zero, by the symmetry of $\widehat{\HFK}(L)$. It follows that $0=n-\sum_{i=1}^m(2-2g-n_i)$, where $n$ is the number of components of $L$, and the surface that $L$ bounds has $m$ components, the $i$th of which has $n_i$ boundary components. Thus $0=n-m+\sum_i g_i$. But $n\geq m$, so $0=n-m+\sum_{i=1}^m g_i\geq \sum_{i =1}^{m}g_i\geq 0$. Thus $\sum_{i=1}^m g_i=0$ and $n=m$. It follows that $L$ bounds $n$ disks, whence $L$ is an unlink with $n$ components.
\end{proof}

Throughout this section we use $V$ to denote the rank two vector space of multi-Alexander grading $0$, with one generator of Maslov index $0$ and the other $-1$.

\begin{proof}[Proof of Lemma~\ref{unlinkdetection}]
    We proceed by induction. For the base, $n=0$, the result follows from the fact that knots of genus at least one have rank at least two, since they have rank at least one in Alexander gradings $g$ and $-g$, where $g$ is the genus of the knot. For the inductive step, note that if $L$ is an $n+1$ component link with $\rank(\widehat{\HFK}(L))=2^n$ then for any component $L_i$ of $L$, $L-L_i$ is a link with $\rank(\widehat{\HFK}(L-L_i))=2^{n-1}$. By inductive hypothesis $L-L_i$ is the $n$ component unlink. It follows that the support of $\widehat{\HFL}(L)$ is contained in the line defined by $A_j=0$ for $j\neq i$. But $i$ is arbitrary, so $\widehat{\HFL}(L)$ is supported in multi-Alexander grading $\mathbf{0}$, and the result follows from Lemma~\ref{alexandergradingcharachterizationofunlinks}.
\end{proof}

We now prove Theorem~\ref{rank4}.

\begin{proof}[Proof of Theorem~\ref{rank4}]
Let $L$ be an $n$ component link with $\rank(\widehat{\HFK}(L))=4$. Since $\widehat{\HFK}(L)$ admits a spectral sequence to $\widehat{\HF}(\#^{n-1}S^1\times S^2)$, it follows that $n\leq 3$. Indeed, since $\rank(\widehat{\HFK}(L))$ is even, it follows that $L$ has either two or three components. We evaluate each case separately.

If $L$ has three components then $L$ is the three component unlink by lemma~\ref{unlinkdetection}. It therefore remains only to consider the case that $n=2$. Let $g$ be the maximal Alexander grading of $\widehat{\HFK}(L)$. We consider the following three cases: $\rank(\widehat{\HFK}(L)(g))$ is either four, two or one, in which case the link is fibered by \cite{ni2007knot}.\begin{enumerate}
       
       \item{\textbf{$\rank(\widehat{\HFK}(L)(g))=4$}.} In this case $g=0$, whence $\widehat{\HFK}(L)$ does not admit a spectral sequence to $\widehat{\HF}(S^1\times S^2)$.
          \item{\textbf{$\rank(\widehat{\HFK}(L)(g))=2$}.} In this case we see that 
         $$\widehat{\HFK}(L)\cong (\F_{m_1}\oplus\F_{m_2})[g]\oplus (\F_{m_1-2g}\oplus\F_{m_2-2g})[-g].$$
        
        Observe that $g>0$. Consider the spectral sequence from $\widehat{\HFK}(L)$ to $\widehat{\HF}(S^1\times S^2)$. There must be exactly one generator with non-trivial differential. Without loss of generality we may then take $m_1-1=m_2-2g$, but in that case we cannot have that the remaining generators from the set $\{\frac{1}{2},\frac{-1}{2}\}$, a contradiction.
        
       \item{\textbf{$\rank(\widehat{\HFK}(L)(g))=1$}.} Observe that $g\geq 1$. Let $[x,g-x]$ be the bi-Alexander grading of the contact class in $\widehat{\HFL}(L)$.
       
       If $g>1$, then Theorem~\ref{rankinnequalityforfiberedlinks} implies that $\rank(\widehat{\HFK}(L))\geq 6$. Thus $g=1$, and $$\widehat{\HFL}(L)\cong \F_{m}(\frac{1}{2},\frac{1}{2})\oplus \F_{m-1}(-\frac{1}{2},\frac{1}{2})\oplus \F_{m-1}(\frac{1}{2},-\frac{1}{2})\oplus \F_{m-2}(-\frac{1}{2},-\frac{1}{2})$$
       Since $\widehat{\HFL}(L)$ admits a spectral sequence to $V$ it follows that $m\in\{0,1\}$. These two cases correspond to the two oriented link types of $T(2,2)$, which link Floer homology is known to detect per~\cite{BG}.
    \end{enumerate}

\end{proof}

We now prove Theorem~\ref{rank6}, proceeding with a method different from the rank four case: we first analyse the case that $L$ is fibered under some orientation, and the remaining cases separately.

\begin{lemma}\label{rank6fibered}
There is no fibered link $L$ such that $\rank(\widehat{\HFK}(L))=6$.
\end{lemma}

\begin{proof}[Proof of Lemma~\ref{rank6fibered}]
Suppose $L$ is an $n$ component fibered link with $\rank(\widehat{\HFK}(L))=6$. Since $\rank(\widehat{\HFL}(L))$ is even, $L$ cannot be a knot. Moreover, $L$ can have at most $3$ components, since $\widehat{\HFL}(L)$ admits a spectral sequence to $\widehat{\HF}(\#^{n-1}S^1\times S^2)$.

If $L$ has three components, then since $L$ is fibered, the maximal Euler characteristic Seifert surface for $L$ is connected, whence $\chi(L)\leq -1$, and the maximal Alexander grading in which $\widehat{\HFK}(L)$ has support is at least $2$. By Theorem~\ref{rankinnequalityforfiberedlinks}, the next to top Alexander grading is of rank three, whence $\widehat{\HFK}(L)$ is of rank at least eight, a contradiction.

Thus $L$ has two components. Again, as $L$ is fibered, the maximal Euler characteristic Seifert surface for $L$ is connected, whence $\chi(L)\leq 0$, and the maximal Alexander grading, $g$, in which $\widehat{\HFK}(L)$ has support is at least $1$.

If $g=1$ then $L$ bounds an annulus, and hence $L$ is a $(2,2n)$-cable. Note that each component of $L$ has knot Floer homology of rank at most four and is therefore an unknot or a trefoil. None of the links $T(2,2n)$ has knot Floer homology of rank $6$, while if $L$ is the $2$-cable of a trefoil, we can see that the linking number of $L$ must be zero since the spectral sequence to either component must be trivial, since $\rank(\widehat{\HFK}(L))=\rank(\widehat{\HFK}(T(2,\pm 3)\otimes V)$. However, we can see that in the top $A_1$ or $A_2$ grading the link Floer homology is of rank $2$ whence at least one component is braided about the other -- contradicting the fact that the linking number is zero -- or $L$ is a split link, which is a contradiction since split links are not fibered.

Suppose now that $g>1$. Let $[x,g-x]$ be the bi-Alexander grading of the top Alexander grading generator. Lemma~\ref{generalnonzero} implies that the complex is given by

\[\F_m[x,g-x]\oplus\F_{m-1}[x-1,g-x] \oplus \F_{m-1}[x, g-x-1]\] \[ \oplus \F_{m-2g}[-x,x-g]\oplus\F_{m-2g+1}[-x+1,x-g] \oplus \F_{m-2g+1}[-x, x-g+1]\]
         
         To ensure that each Alexander grading is of even rank we must have that $g=2, x=1$, i.e. the complex is 
         $$\F_m[1,1]\oplus\F_{m-1}[1,0]\oplus\F_{m-1}[0,1]\oplus\F_{m-4}[-1,-1]\oplus\F_{m-3}[-1,0]\oplus\F_{m-3}[0,-1] $$

             By the symmetry of $\widehat{\HFL}(L)$, we can see that the link obtained by reversing the orientation of one of the complements bounds an annulus, or the disjoint union of a disk and a punctured torus. Since split links cannot be fibered, the latter option is impossible.
         
          Suppose $L$ bounds an annulus. Again since there is no torus link $T(2,2n)$ of rank $6$, $L$ must have trefoil components. Since the spectral sequence corresponding to either component is trivial, it follows that the linking number is zero. But from the complex we can see that either $L$ is split -- which is again impossible since we are assuming $L$ is fibered -- or the components are exchangeably braided by braid detection \cite{martin2022khovanov}, a contradiction.

\end{proof}

Having eliminated the case that $L$ is fibered, we can proceed to the remaining cases.

\begin{proof}[Proof of Theorem~\ref{rank6}]
Let $L$ be an $n$ component link with link Floer homology of rank $6$. Lemma~\ref{rank6fibered} implies that $L$ is not fibered under any orientation. The fact that $2^{n-1}\leq\rank(\widehat{\HFK}(L))$ implies that $L$ has at most three components. Since $\rank(\widehat{\HFK}(L))$ is even, it follows that $L$ has two or three components.

Consider the Link Floer homology polytope of $L$. Observe that it is contained in the polytope defined by $\sum_{j=1}^n A_j=\pm g_i$ for each $g_i$, where $g_i$ is maximal Alexander grading of $\widehat{\HFK}(L)$ under some orientation of $L$.

Lemma~\ref{alexandergradingcharachterizationofunlinks}, together with the fact that there is no unlink with Knot Floer homology of rank $6$, implies that we may take $g_j>0$ for all $i$.

Suppose that $\widehat{\HFL}(L)$ has support at a vertex $p=(p_1,p_2,\dots p_n)$ of the link Floer polytope. Since there exists an Alexander grading $A_i$ for which the support of $\widehat{\HFL}(L)$ in the plane $A_i=p_i$ is exactly $p$, it follows that the spectral sequence from $\widehat{\HFL}(L)$ to $\widehat{\HFL}(L_i)\otimes V^{n-1}$ is trivial. Thus the rank of $\widehat{\HFL(L)}$ at $p$ is at least $2^{n-1}$. Since the polytope is non-degenerate, the symmetry properties of $\widehat{\HFL}(L)$ implies that $2^n\leq 6$, whence $n=2$. The linking number is zero, as the range of the Alexander gradings cannot change under the spectral sequence to the relevant component, because of symmetry properties of knot Floer homology of knots.  The symmetry of $\widehat{\HFL}(L)$ implies that the rank at $p$ of the Link Floer homology is four so that $L$ is a split link -- and the only such link with link is the disjoint union of an unknot and a trefoil -- or $L_i$ is a braid axis -- which is impossible since the linking number of $L_i$ and the other component is zero. 

It thus remains only to consider the case that $\widehat{\HFL}(L)$ does not have support at any of the vertices. Since $L$ is not fibered, each face of the polytope must have at least two generators. If $L$
 has two components, then it must have at least four faces, each of which has two generators in its interior, a contradiction. If $L$ has three components, notice that the polytope has a collection of four faces that meet only at vertices. Since the polytope does not have support at the vertices, it follows that $\rank(\widehat{\HFL}(L))\geq 8$, a contradiction. \end{proof}

For the rank eight case, we proceed as in the rank six case, dealing with fibered links first, and the remaining cases subsequently.

\begin{proposition}\label{rank8fibered}
Suppose $L$ is fibered and $\rank(\widehat{\HFK}(L))=8$. Then $L\cong T(2,4)$ or $T(2,-4)$.
\end{proposition}
\begin{proof}
Let $L$ be a fibered link with link Floer homology of rank eight. Since $\widehat{\HFK}(L)$ admits a spectral sequence to $\widehat{\HF}(\#^{n-1}S^1\times S^2)$, where $n$ is the number of components of $L$, it follows that $n$ has at most four components. Indeed, $L$ cannot be a knot since $\rank(\widehat{\HFK}(L))$ is even. Suppose $L$ has four components. Then by Theorem~\ref{unlinkdetection}, $L$ must be a four component unlink, which is not fibered.

Suppose $L$ has three components. Since $L$ is fibered, the minimal genus Seifert surface for $L$ is the fiber surface, which is connected. It follows that the highest Alexander grading, $g$, in which $\widehat{\HFK}(L)$ has supported is at least $2$. By Lemma~\ref{generalnonzero}, $\rank(\widehat{\HFK}(L)(g-1))\geq3$. Thus $$\widehat{\HFK}(L)\cong\F_m[g]\oplus\F_{m-1}^3[g-1]\oplus \F_{m+1-2g}^3[1-g]\oplus\F_{m-2g}[-g]$$ where $m\in\Z$. It follows that the Alexander polynomial of this link does not evaluate to zero on $1$, a contradiction.

Hence $L$ is a two component link. Then, as previously, the maximal Alexander grading, $g$, of $\widehat{\HFK}(L)$ is at least one. Suppose $g=1$. Then $\rank(\widehat{\HFK}(L;0))=6$. Such a complex clearly does not admit a spectral sequence to $\widehat{\HF}(S^1\times S^2)$, a contradiction. Thus $g\geq 2$. Consider $\widehat{\HFL}(L)$. Suppose the top Alexander grading generator is in multi-Alexander grading $(x,g-x)$ and of Maslov index $m$. Lemma~\ref{generalnonzero} implies that there are generators in Alexander gradings $(x-1,g-x)$, $(x,g-x-1)$ both of Maslov index $m-1$, while symmetry implies that there are generators in multi-Alexander grading $(-x,x-g)$ of Maslov index $m-2g$ as well as in bi-Alexander gradings $(1-x,x-g)$ and $(-x,g+1-x)$ of Maslov index $m+1-2g$. Since $g\geq 2$ it follows that all of these generators are distinct. To ensure that the complex is even rank in each $A_i = k$ grading, we require that there are generators in multi-Alexander gradings $(x-1,g-x-1)$ and $(1-x,x+1-g)$. Observe that the symmetry properties of link Floer homology, together with the existence of the spectral sequences to complexes arising as tensor products with $V$, imply that the Maslov gradings of these generators are $m-2$ and $m+2-2g$ respectively.

Observe that each component of $L$ has rank at most four whence each component is an unknot or a trefoil.

Without loss of generality, after relabeling we may take $2x\geq g$. Suppose $(x,g)\not\in \{(\frac{3}{2},2),(1,2)\}$. Then the chain complex splits, as the sum of two squares, and each component is an unknot, since there do not exist three consecutive Alexander gradings containing generators of the correct Maslov indices to have a trefoil component. To ensure the spectral sequence from $\widehat{\HFK}(L)$ to $\widehat{\HFK}(L_1)\otimes V$, we must add three horizontal arrows, meaning that one of the squares gets at least two. Correspondingly we must add three vertical arrows, and to insure that $\partial^2=0$ we must then have that the two arrows are added to the square with two horizontal arrows. This means that the upper right corner, or lower left corner of one of the squares persists under both spectral sequences and hence must be on the diagonal, as its coordinates are then given by $(\frac{\lk(L_1,L_2)}{2},\frac{\lk(L_1,L_2)}{2})$. Reversing the orientation of a component, we see that $L$ bounds an Euler characteristic zero surface, that is either an annulus or the disjoint union of a trefoil an unknot. Since fibered links cannot be split, $L$ bounds an annulus, and since each component is unknotted, it follows that $L$ is $T(2,2n)$ for some $n$. The only $n$ for which $\rank(\widehat{\HFK}(T(2,2n))=8$ are $n=\pm2$.

Suppose $(x,g)=(1,2)$. Reversing the orientation of a component, yields a link that bounds an annulus (or the disjoint union of a punctured torus and a disk, but such a link is split and therefore not fibered) whence $L$ is a $2$-cable link. $L$ cannot be the $2$-cable of a trefoil, as we know that the linking number is $0$, since the span of each $A_i$ grading is $3$, but the rank of $\widehat{\HFL}(L)$ in the top Alexander $A_1$ grading is $2$ telling us that $L$ is either split or each component is a braid axis, which would contradict the fact that $L$ is fibered. $L$ cannot be split since split links aren't fibered. If $L$ is a $(2,2n)$-cable of the unknot, then $L$ is a torus link $T(2,2n)$ for some $n$, but the only such torus links with link Floer homology of rank $8$ are $T(2,4)$ and $T(2,-4)$, as required.

Suppose $(x,g)=(\frac{3}{2},2)$. It can be quickly checked that this complex does not admit the requisite spectral sequences, irrespective of whether the remaining component is a trefoil or an unknot. Specifically, it can readily be seen that there cannot be a component of the differential from a generator in $A_1$ grading $\frac{1}{2}$ to one in lower $A_1$ grading, as else the generators that persist under the spectral sequence to $\widehat{\HFL}(L_2)$ do not have admissible Maslov gradings. Thus the complex splits as two squares, and again must be located on the diagonal, a contradiction.

\end{proof}

\begin{proof}[Proof of Theorem~\ref{rank8}]

By the previous lemma it suffices to classify links that are not fibered under any orientation and with link Floer homology of rank eight. Suppose $L$ is such a link, with $n$ components. Since $\rank(\widehat{\HFK}(L))\geq 2^{n-1}$ and $L$ is even, $L$ has either two, three or four components. If $L$ has four components, then Lemma~\ref{unlinkdetection} implies that $L$ is an unlink. We deal with the two remaining cases separately:\begin{enumerate}
   
    \item\textbf{$L$ has two components}
Consider $\widehat{\HFL}(L)$. Since $L$ is not fibered under any orientation, the link Floer polytope of $L$ is a subset of the polytope with boundaries $a_1+a_2=g_1$, $-a_1-a_2=g_1$, $a_1-a_2=g_2$, $a_2-a_1=g_2$. Here $g_1,g_2$ are the maximal Alexander gradings of $\widehat{\HFK}(L)$ under its two possible orientations. This has vertices $\pm(\frac{g_1+g_2}{2},\frac{g_1-g_2}{2}), \pm(\frac{g_1-g_2}{2},\frac{g_1+g_2}{2})$.

Lemma~\ref{alexandergradingcharachterizationofunlinks} together with the fact that the two component unlink has knot Floer homology of rank two imply that we may take $g_1,g_2>0$.

Suppose that $\widehat{\HFL}(L)$ has support at one of the vertices of the polytope. After relabeling the components, and considering the symmetry of the complex, we may take them to be a generators in $(\frac{g_1+g_2}{2},\frac{g_1-g_2}{2})$, $(-\frac{g_1+g_2}{2},\frac{g_2-g_1}{2})$. Observe that both of the generators must persist under the spectral sequence to $\widehat{\HFL}(L_1)$, whence $g(L_1)=g_1+g_2$. Since each component of $L$ is either an unknot or a trefoil it follows that $g_1+g_2\in\{0,1\}$. In either case it follows that at least one of $g_1$ or $g_2$ is zero, a contradiction.

Thus the polytope defined does not have support at its vertices. It follows that the interior of each face contains exactly two generators, see Figure~\ref{rank2face} for an example. Consider the $a_1+a_2=g_1$ face. Assume the two generators on the interior of this face are of distinct multi-Alexander gradings. In order that each $A_1,A_2$ grading is of even rank and the complex admits the requisite symmetries and spectral sequences we must have that $g_1=g_2$. Indeed, considering the Maslov grading $m$ of one of the generators on the $A_1+A_2=g_1$ face we see that $m-2(g_1+g_2)=m-2$, whence again $g_1+g_2=1$, and we obtain a contradiction as previously, see Figure~\ref{rank2face}.

	\begin{figure}[h]
    \centering
    \includegraphics[width=7cm]{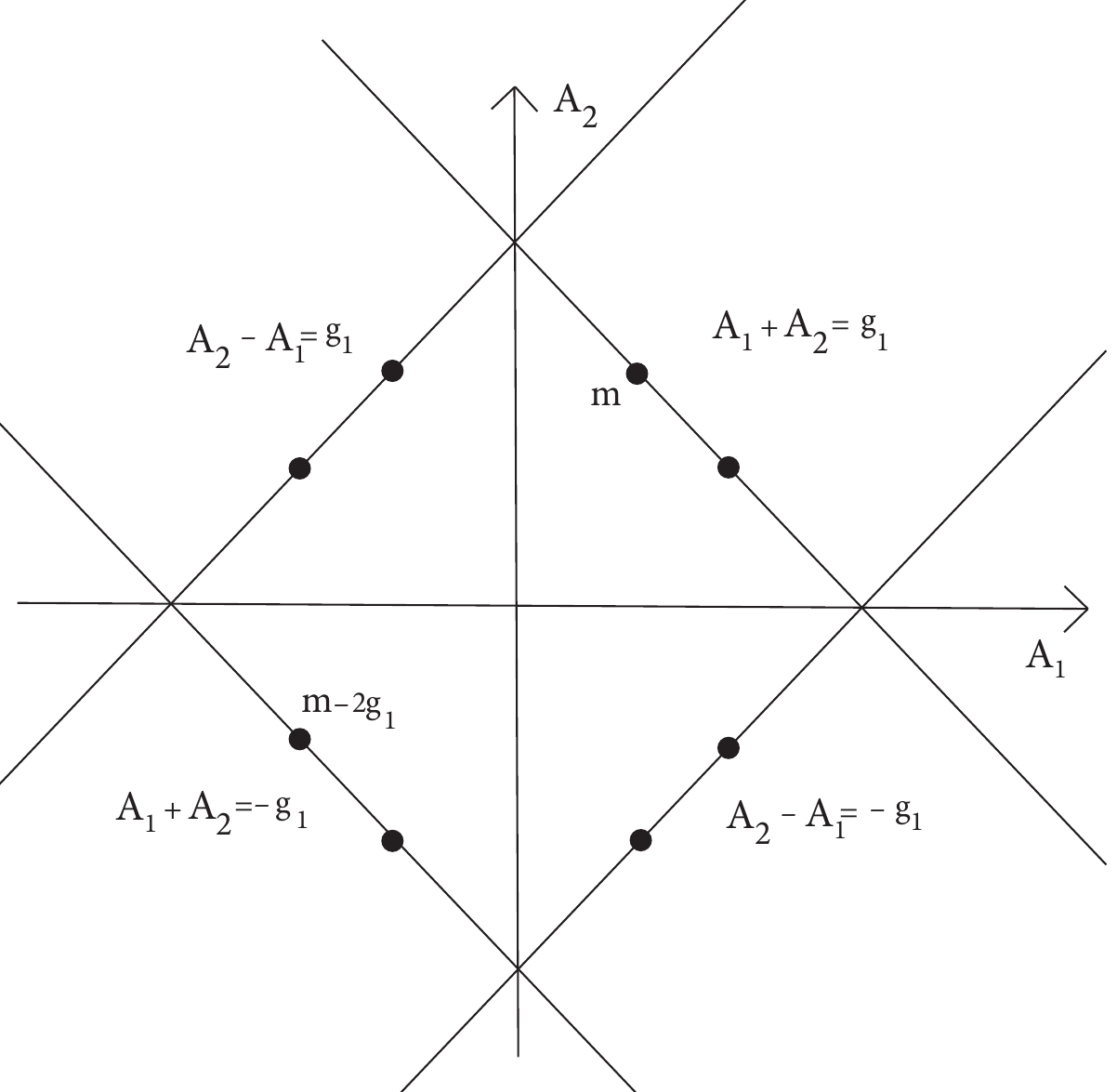}
    \caption{A rank eight bigraded vector space not arising as the link Floer homology of a link.\label{rank2face}}
    \end{figure}

The only remaining case is that on each face the two generators sit in the same multi-Alexander grading, see Figure~\ref{rank2face}. Suppose the generators occur in multi-Alexander gradings $\pm(x,g_1-x)$, $\pm(y,g_2+y)$. In order that the spectral sequences to each component are non-zero, in which case the rank of one the components would be even, we must have that $x=-y$, $g_1-x=g_2+y$. It follows that $g_1=g_2$. Since each $\widehat{\HFL}(L)$ has support in only two $A_i$, $A_j$ gradings, it follows that each component is an unknot. Without loss of generality we may take the linking number to be positive, whence it follows that $x=g_1-x$. It follows that the complex has support in grading $\{\pm(\frac{g_1}{2},\frac{g_1}{2}),(\pm\frac{g_1}{2},\mp\frac{g_1}{2})\}$. Let $m_1\geq m_2$ be the Maslov gradings of the two generators of grading $(\frac{g_1}{2},\frac{g_1}{2})$. Since $A_1$ is an unknot we must have that $\partial_1 x\neq 0$. Without loss of generality, let $\partial_1x=y$. Since $y$ does not persist under the spectral sequence to $\widehat{\HFL}(L_2)$, $\partial_2 y=z$. Thus $z$ is of Maslov grading $m-2$, so that either $g_1=1$, in which case $L$ bounds an annulus and is thereby $T(2,2n)$ for some $n$ a contradiction, since these do not have the right homology, or the remaining generator of bi-Alexander grading is of Maslov grading $\{0,1\}\ni m+2g_1-2\geq m+2$. The remaining generator of bi-Alexander grading $(\frac{g_1}{2},\frac{-g_1}{2})$ must be of Maslov grading $m+2g_1-2\pm 1$. This is because the other remaining generator of bi-Alexander grading $(-\frac{g_1}{2},-\frac{g_1}{2})$ is of Maslov grading either $m+2g_1-2$ or $m+2g_1-2$. But it should also be of grading $m-2g_1$, by symmetry, a contradiction.


\item \textbf{$L$ has three components}
Suppose $L$ has three components. Then each two component sublink is either a Hopf link or a two component unlink. Suppose a two component sublink $L-L_3$ is a Hopf link. Then the linking number of $L_3$ with $L_1$ and $L_2$ must be zero as else the rank of the complex would be greater than eight. It follows that the spectral sequence from $\widehat{\HFL}(L)$ to $\widehat{\HFL}(L-L_3)$ is just projection onto the $A_3=0$ plane. Suppose $\widehat{\HFL}(L)$ is supported in the $A_3=0$ plane, then $L$ consists of a disjoint union of a Hopf link and an unknot. If $\widehat{\HFL}(L)$ is not supported in the $A_3=0$ plane then since the rank in the top $A_1$ grading is four, it follows that $L_1$ is a braid axis -- which would contradict the fact that the linking number is non-zero -- or $L$ is split, in which case $L$ is the disjoint union of the a Hopf link and an unknot.

Suppose now that each two component sublink consists of a two component unlink. It follows that the $\widehat{\HFL}(L)$ has rank four on each plane $\{A_i=0\}_{i=1,2,3}$. $L$ must be of rank eight in multi-Alexander grading $(0,0,0)$, as else the rank is too high. To see this note that if there is a generator in multi-Alexander grading $(0,0,x)$ for $x>0$, then there must also exist generators in multi-Alexander gradings $(0,y_1,x)$, $(y_2,0,x)$ to ensure that the generator in multi-Alexander grading $(0,0,x)$ do not persist under the spectral sequences to $L-L_1$, $L-L_2$. It follows that there are at least $6$ generators that have non-zero $A_3$ component, and at least another $4$ with $A_3=0$, a contradiction.

The result then follows from Lemma~\ref{alexandergradingcharachterizationofunlinks}.
\end{enumerate}

\end{proof}

\begin{proof}[Proof of Proposition~\ref{3componentrank12}]
Let $L$ be a $3$ component link with $\rank(\widehat{\HFK}(L))=10$. Suppose $L$ is fibered. The maximal Alexander grading, $A$, of $\widehat{\HFK}(L)$ is given by $2+g$ for some $g\geq 0$. Note that $\rank(\widehat{\HFK}(L)(A))=1$. Suppose the generator is of Maslov index $m$. Theorem~\ref{rankinnequalityforfiberedlinks} implies that there are at least three generators of Maslov index $m-1$ in Alexander grading $A-1$. The symmetry of $\widehat{\HFK}(L)$ yields a summand $\F_{m-4-2g}[-2-g]\oplus\F_{m-3-2g}^3[-1-g]$. In order for the Alexander polynomial evaluate to $0$ on $-1$, we must add at least another four generators, yielding a complex of rank at least $12$, a contradiction.

We may thus assume that $L$ is not fibered under any orientation. Let $g_i$ be the maximal Alexander grading of $\widehat{\HFK}(L)$ under each of the four possible orientations of $L$, up to an overall reversal in the orientation of $L$. Note that $g_i>0$ by Lemma~\ref{alexandergradingcharachterizationofunlinks}. Consider the polytope $P$ cut out by $A_1+A_2+A_3=\pm g_1$, $A_1+A_2-A_3=\pm g_2$, $A_1-A_2+A_3=\pm g_3$, $A_1-A_2-A_3=\pm g_4$. Observe that each of these faces must contain at least two generators. If the edges of $P$ do not contain any generators, then $\rank(\widehat{\HFK}(L))\geq 16$, a contradiction.

We may therefore assume a generator lies on an edge. Without loss of generality we may take this to be the edge formed by $A_1+A_2+A_3=g_1$, $A_1+A_2-A_3=g_2$, namely $A_1+A_2=\dfrac{g_1+g_2}{2}$. Observe that any generator on this edge, as well as the respective generators yielded by the symmetry of $\widehat{\HFL}(L)$ on the edge $-A_1-A_2=\dfrac{g_1+g_2}{2}$, persists under the spectral sequence to $\widehat{\HFL}(L-L_3)\otimes V[\frac{\lk(L_1,L_3)}{2},\frac{\lk(L_2,L_3)}{2}]$, and is indeed of maximal (respectively minimal $A_1+A_2$ grading). Now, there are only two 2 component links with $\rank(\widehat{\HFK}(L))\leq 5$; the Hopf link and the unlink. It follows that $\lk(L_i,L_j)\in\{-1,0,1\}$ for each $i,j$. Thus the maximal $A_1+A_2$ grading can only be $1$. That is, we have that $1=\dfrac{g_1+g_2}{2}$, so that $g_1=g_2=1$. It follows that, under an appropriate orientation of $L$, $\chi(L)=1$. It follows that $L$ bounds either an annulus and a disk, or two disks and a punctured torus. Note that for each $i$, $\rank(\widehat{\HFK}(L-L_i)\leq 5$, so that $L-L_i$ is an unlink, or a Hopf link. Thus $L$ is the disjoint union of a Hopf link and an unknot, but this link has knot Floer homology of rank eight, a contradiction.

\end{proof}

We conclude with a proof of  Theorem~\ref{HFKdetcetstrefoilmeridian}

\begin{proof}[Proof of Theorem~\ref{HFKdetcetstrefoilmeridian}]

Let $L_0$ be the trefoil together with a meridian. It suffices to show that if  $\widehat{\HFK}(L)\cong\widehat{\HFK}(L_0)$ then  $\widehat{\HFL}(L)\cong\widehat{\HFL}(L_0)$ by the corresponding link Floer homology detection result in~\cite{binns2020knot}.

Since the maximal Maslov grading of $\widehat{\HFK}(L)$ is $\frac{1}{2}$, it follows that $L$ has at most two components. Since $\rank(\widehat{\HFK}(L))$ is even, it follows that $L$ cannot be a knot, whence $L$ is a two component link.

Observe that $\widehat{\HFK}(L)$ detects the linking number of two component links, so the linking number is one. Consider $\widehat{\HFL}(L)$. Let $A_i(0)$ be the Alexander grading of the Maslov index $0$ generator in $\widehat{\HFL}(L_i)$. Since there are no positive Maslov index generators in $\widehat{\HFL}(L)$, it follows that $A_i(0)\geq 0$. Since $0\leq A_1(0)+A_2(0)+\lk(L_1,L_2)=2$, we have that $\{A_1(0),A_2(0)\}=\{0,1\}$. Without loss of generality we may take $A_1(0)=1$. It follows that there is a Maslov index $0$ generator in bi-Alexander grading $(\frac{3}{2},\frac{1}{2})$. By Lemma~\ref{generalnonzero}, we have that there are Maslov index $-1$ generators of bi-Alexander grading  $(\frac{1}{2},\frac{1}{2}),(\frac{3}{2},\frac{-1}{-2})$. Observe that in order that $\widehat{\HFL}(L)$ admit a spectral sequence to $\widehat{\HFL}(L_1)\otimes V$, there must be a Maslov index $-2$ generator with $A_1$ grading $\frac{1}{2}$. Such a generator must be in bi-Alexander grading $(\frac{1}{2},\frac{-1}{2})$. By symmetry this determines the Alexander gradings of all but four generators. Consider the remaining Maslov index $-1$ suppose it is of bi-Alexander grading $(x,1-x)$. Observe that in order to admit spectral sequences to $\widehat{\HFL}(L_i)\otimes V$, there must be Maslov index $-2$ generators in bi-Alexander  gradings $(x-1,1-x)$ and $(x,-x)$. In order that $\widehat{\HFL}(L)$ be symmetric, we must then have that $x=\frac{1}{2}$. We have thus determined that $\widehat{\HFL}(L)\cong\widehat{\HFL}(L_0)$, as desired.
\end{proof}

\end{section}

\begin{section}{Detection Results for Khovanov Homology}
\label{Detection Results for Khovanov Homology}

In this section we briefly review Khovanov homology before proving various detection results.

\begin{subsection}{A Brief Review of Khovanov homology}

Let $L$ be a link, and $R$ be the ring $\Z,\Z_2$ or $\Q$. The \emph{Khovanov chain complex of $L$}, $(\CKh(L,R),\partial)$, is a finitely generated $\Z\oplus\Z$-filtered chain complex over $R$~\cite{khovanov2000categorification}.

$$\CKh(L;R):=\underset{i,j\in\Z}{\bigoplus}\CKh^{i,j}(L)$$

Here $i$ is called the \textit{homological grading}, while $j$ is called the \textit{quantum grading}. The filtered chain homotopy type of $L$ is an invariant of $L$. The $R$-module $\CKh(L,R)$ has generators corresponding to decorated complete resolutions of $D$, while $\partial$ is determined by a simple TQFT. The parity of the $j$ gradings in which $Kh(L)$ has non-trivial support agree with the parity of the number of components of $L$. The \textit{Khovanov homology of $L$} is obtained by taking the homology of $\CKh(L;R)$. A choice of basepoint $p\in L$ induces an action on $\CKh(L;R)$, which commutes with the differential. Taking the quotient of $\CKh(L;R)$ by this action and then taking homology yields a bigraded $R$-module called the~\textit{reduced Khovanov homology} of $L$, denoted $\widetilde{\Kh}(L,p;R)$~\cite{khovanov2003patterns}. Given a collection of points $\mathbf{p}=\{p_1,p_2,\dots,p_k\}\subset L$, there is a generalization of reduced Khovanov homology called \textit{pointed Khovanov homology}, $\Kh(L,\mathbf{p};R)$ due to Baldwin-Levine-Sarkar~\cite{baldwin2017khovanov}.

$\Kh(L;R)$ admits a number of useful spectral sequences. Suppose $L$ has two components $L_1,L_2$. If $R$ is $\Q$ or $\Z_2$ then $\Kh(L;R)$ admits a spectral sequence to $\Kh(L_1;R)\otimes\Kh(L_2;R)$, called the \emph{Batson-Seed spectral sequence}~\cite{batson2015link}. Indeed this spectral sequences respects the $i-j$ grading on $\Kh(L;R)$ in the sense that:
\begin{align}\label{GradedBatsonSeed}
    \rank^{i-j=l}(\Kh(L;R)\geq \rank^{i-j=l+2\lk(L_1,L_2)}(\Kh(L_1;R)\otimes\Kh(L_2;R))
\end{align}

 There is another spectral sequence, called the \emph{Lee spectral sequence}, from $\Kh(L;\Q)$ which abuts to an invariant called the \emph{Lee Homology} of $L$, $\LLL(L):=\bigoplus_i\LLL^i(L)$~\cite{lee2005endomorphism}. This spectral sequence respects the $i$ gradings in the sense that $\rank(\Kh^i(L;\Q))\geq \rank(\LLL^i(L))$. Lee showed that $\LLL(L)\cong\bigoplus_{i=1}^n\Q^2_{a_i}$ where $a_i$ are integers, where $L$ has $n$ components. Indeed if $L$ has two components $L_1,L_2$ then $\LLL(L)\cong\Q^2_{0}\oplus\Q^2_{\lk(L_1,L_2)}$~\cite[Proposition 4.3]{lee2005endomorphism}. Moreover, there is an invariant, $s(L)$ -- due to Rasmussen~\cite{Rasmussen} in the knot case and Beliakova-Wehrli in the link case~\cite{beliakova2008categorification} -- defined from the Lee spectral sequence -- with the property that $s(L)\leq 1 -\chi(L)$.

Finally, there is a spectral sequence, due to Dowlin~\cite{dowlin2018spectral}, relating Khovanov homology and Knot Floer homology. If $L$ is a link and $\mathbf{p}\subseteq L$, with exactly one element of $\mathbf{p}$ in each component of $L$, then there is a spectral sequence from $\Kh(L,\mathbf{p};\Q)$ to $\widehat{\HFK}(L;\Q)$ that respects the relative $\delta$-gradings. Here $\widehat{\HFK}(L;\Q)$ uses the coherent system of orientations given in~\cite{alishahi2015refinement}. We use this version of knot Floer homology, and the corresponding link Floer homology for the remainder of the paper. As a corollary, Dowlin shows that if $L$ has $n$ components then: \begin{align}
    2^{n-1}\rank(\widetilde{\Kh}(L;\Q))\geq \rank(\widehat{\HFK}(L;\Q))\label{DowlinReduced}\end{align}

\end{subsection}

\begin{subsection}{Rank Detection Results for Khovanov homology}

Applying Dowlin's spectral sequence and the rank detection results for knot Floer homology from Section~\ref{Rank Detection Results for Knot Floer homology}, we have the following result:
\begin{corollary}\label{KHr2components}
Suppose $L$ is a two component pointed link with $\rank(\widetilde{\Kh}(L,p;\Z_2))\leq 4$. Then $L$ is one of:\begin{itemize}
    \item an unlink;
\item a Hopf link;
\item $T(2,4)$ or $T(2,-4)$.
\end{itemize}
\end{corollary}

Since $\rank(\Kh(L;\Z_2))=2\rank(\widetilde{\Kh}(L;\Z_2))$ by \cite{shumakovitch2014torsion} and $ \rank(\widetilde{\Kh}(L;\Z_2))\geq\rank(\widetilde{\Kh}(L;\Q))$ by the universal coefficient theorem, this recovers Zhang-Xie's classification of two component links with $\rank(\Kh(L;\Z_2))\leq 8$~\cite[Corollary 1.2]{xie2022links}.

\begin{corollary}\label{KHr3components}
Suppose $L$ is a three component pointed link. Then $\rank(\widetilde{\Kh}(L,p;\Q))> 2$.
\end{corollary}

\begin{proof}[Proof of Corollary~\ref{KHr2components}]
Suppose $L$ is a two component link with $\rank(\widetilde{\Kh}(L;\Q))\leq 4$. Equation~\ref{DowlinReduced} implies that $\rank(\widehat{\HFK}(L;\Q))\leq 8$, and the result follows from Theorems~\ref{rank4}, \ref{rank6} and \ref{rank8}.
\end{proof}

\begin{proof}[Proof of Corollary~\ref{KHr3components}]
Suppose $L$ is a three component link with $\rank(\widetilde{\Kh}(L;\Q))\leq 2$. Equation~\ref{DowlinReduced} implies that $\rank(\widehat{\HFK}(L;\Q))\leq 8$. Theorems~\ref{rank4}, \ref{rank6} and \ref{rank8} imply that $L$ is either the disjoint union of a Hopf link and an unknot, or an unlink. But these links have reduced Khovanov homology of too high a rank, a contradiction.
\end{proof}

\end{subsection}

\begin{subsection}{Khovanov Homology Detects $T(2,8)$, $T(2,10)$}
We prove the following theorems:

	\begin{theorem}\label{KHTT28}
		Suppose $\Kh(L;\Z)\cong\Kh(T(2,8);\Z)$. Then $L$ is isotopic to $T(2,8)$
	\end{theorem}
	
\begin{theorem}\label{KHT210}
	Suppose $\Kh(L;\Z)\cong\Kh(T(2,10);\Z)$. Then $L$ is isotopic to $ T(2,10)$.
\end{theorem}

\[ \begin{array}{c|c c c c c c c c c}
\label{Table 1}
24&&&&&&&&&\Z\\22&&&&&&&&\Z&\Z\\
 20&&&&&&&&\Z_2
 \\
 
 18&&&&&&\Z&\Z
 \\
 
 16&&&&&&\Z_2
 \\
 14&&&&\Z&\Z\\
 
 12&&&&\Z_2\\
 
 10&&&\Z\\
 
 8&\Z\\
 
 6&\Z\\
\hline
&0&1&2&3&4&5&6&7&8

\end{array}
\]

Table 1. $\Kh(T(2,8);\Z)$: The horizontal axis indicates the $i$ grading, while the vertical axis denotes the $j$ grading

\[\begin{array}{c|c c c c c c c c c c c}
\label{Table 2}
28&&&&&&&&&&&\Z\\
26&&&&&&&&&&\Z&\Z\\
24&&&&&&&&&&\Z_2\\22&&&&&&&&\Z&\Z\\
 22&&&&&&&&\Z_2
 \\
 
 20&&&&&&\Z&\Z
 \\
 
 18&&&&&&\Z_2
 \\
 16&&&&\Z&\Z\\
 
 14&&&&\Z_2\\
 
 12&&&\Z\\
 
 10&\Z\\
 
 8&\Z\\
\hline
&0&1&2&3&4&5&6&7&8&9&10

\end{array}
\]

Table 2. $\Kh(T(2,10);\Z)$: The horizontal axis indicates the $i$ grading, while the vertical axis denotes the $j$ grading

\vspace{.5in}

This requires the Lee, Batson-Seed, and Dowlin spectral sequences, as well as an analysis similar to that of section~\ref{Knot Floer Homology Detects $T(2,8),T(2,10)$}.

$\Kh(T(2,m);\Z)$ was computed in  \cite{khovanov2000categorification}, and is shown in Tables 1 and 2 in the $m=8$ and $m=10$ cases. Note in particular that $\Kh(T(2,2n);\Z)$ is \emph{Khovanov-thin}, meaning that $\Kh(T(2,2n);\Z)$ is supported in two adjacent $j-2i$ gradings.

To prove Theorems~\ref{KHTT28} and~\ref{KHT210} we first show that any link $L$ with $\Kh(L;\Z)\cong\Kh(T(2,2n);\Z)$ has two components, with linking number $n$ and $\chi(L)\geq 2-2n$. We then show that the two components are unknotted in the $n=4,5$ cases. After that we conclude by showing that $L$ must bound an annulus after reversing the orientation of one of the components, thereby completing the proofs.

\begin{lemma}\label{KHcomponentsgenuslinking}
Suppose $\Kh(L;\Z)\cong\Kh(T(2,2n);\Z)$. Then $L$ is a two component link with linking number $n$ and $\chi(L)\leq 2-2n$.
\end{lemma}

We prove this Lemma using Lee's spectral sequence~\cite{lee2005endomorphism}, which has $\Kh(L;\Q)$ as $E_1$ page and $E_\infty$ page a complex of rank $2^{n}$ where $n$ is the number of components of $L$.

\begin{proof}[Proof of Lemma~\ref{KHcomponentsgenuslinking}]
Suppose $L$ is as in the Theorem. By the universal coefficient theorem there are exactly two homological gradings in which the rank of $\Kh^i(L;\Q)$ is at least two, namely $i=0$ and $i=2n$. Since $\rank(\Kh^i(L;\Q))\geq\rank(\LLL^i(L;\Q))$, and $\LLL^i(L;\Q)\cong\bigoplus_{i=1}^m\Q^2_{a_i}$, where $a_i\in\Z$, $m$ is the number of components of $L$, it follows that $L$ can have at most two components. To see that $L$ has exactly two components note that $L$ has an even number of components since $\Kh(L;\Z)$ is supported in even quantum gradings~\cite[Proposition 24]{khovanov2000categorification}.

The linking number of the two components of $L$ is half the difference in the homological gradings of the generators which persist under Lee's spectral sequence ~\cite[Proposition 4.3]{lee2005endomorphism}, hence the linking number of $L$ is $n$.

Finally we note that since $L$ is Khovanov-thin, with generators supported in the gradings $j-2i=2n,2n+2$ it has $s$-invariant given by $n+1$, whence $\chi(L)\leq 2-2n$~\cite{beliakova2008categorification}.
\end{proof}

We can now prove the following two lemmas:

\begin{lemma}\label{KHt228components}
		Suppose $\Kh(L;\Z)\cong\Kh(T(2,8);\Z)$. Then $L$ has unknotted components.
	\end{lemma}

\begin{lemma}\label{KHT2102components}
	Suppose $\Kh(L;\Z)\cong\Kh(T(2,10);\Z)$. Then $L$ has unknotted components.
\end{lemma}

Both of these use the Batson-Seed link splitting spectral sequence~\cite{batson2015link} to bound the rank of the components of $L$ in certain combinations of $i$ and $j$ gradings. We use Corollary 3.4 of~\cite{batson2015link} which states that if $L$ is a link consisting of two components $L_1,L_2$ then
\begin{align}\label{GradedBatsonSeed}
    \rank^l(\Kh(L;F))\geq \rank^{l+2\lk(L_1,L_2)}(\Kh(L_1;F)\otimes\Kh(L_2;F))
\end{align}
 
 Here $F$ is a field -- we will use both $\Z_2$ and $\Q$ accordingly,  and $\rank^l$ denotes the $i-j=l$ grading part of the relevant vector space.
 
 For Lemma~\ref{KHt228components} we use the fact that the only knots $K$ with $\rank(\Kh(K;\Z_2))\leq 8$, are the trefoils and unknot. This follows from Kronheimer-Mrowka's unknot detection result \cite{kronheimer2011khovanov}, as  well as Baldwin-Sivek's trefoil detection results \cite{baldwin2022khovanov} and the fact that for knots $K$ $\rank_{\Z_2}(\Kh(K;\Z_2))$ is of the form $2+4n$ for some $n\geq 0$.
 
 \begin{proof}[Proof of Lemma~\ref{KHt228components}]
Suppose $L$ is as in the hypothesis of the Lemma. Lemma~\ref{KHcomponentsgenuslinking} implies that $L$ consists of two components, $L_1$ and $L_2$ of linking number $4$. From $\Kh(L;\Z)$, as shown in Table 1, we deduce that $\rank_{\Z_2}(\Kh(L;\Z_2))=16$ by the universal coefficient theorem, so the Batson-Seed spectral sequence implies that $\rank(\Kh(L_i;\Z_2))\leq 8$, whence each component of $L$ is an unknot or a trefoil~\cite{baldwin2022khovanov}. Indeed, we readily see that at least one component is unknotted.

	We now consider $\Kh(L;\Q)$, which we can determine from $\Kh(L;\Z)$ via the universal coefficient theorem. Suppose $L$ has a right hand trefoil component. Then \\
	$\rank_{i-j=-4}(\Kh(T(2,3);\Q)\otimes\Kh(U;\Q)))=2$, but $\rank_{i-j=-12}(\Kh(L;\Q))=1$ contradicting equation~\ref{GradedBatsonSeed}. Similarly, $\rank_{i-j=2}(\Kh(T(2,-3);\Q)\otimes\Kh(U;\Q))=3$ which contradicts $\rank_{i-j=-6}(\Kh(L;\Q))=1$ by Equation~\ref{GradedBatsonSeed}. Thus both components of $L$ are unknotted.
\end{proof}

Lemma~\ref{KHT2102components} requires some extra analysis since there is no complete classification of knots with $\rank(\Kh(L;\Z_2))=10$. To deal with this latter case we  use Dowlin's spectral sequence and Proposition~\ref{E2collapsed}.

\begin{proof}[Proof of Lemma~\ref{KHT2102components}]
	Suppose $\Kh(L;\Z)\cong\Kh(T(2,10);\Z)$. Then from Table 2 and the universal coefficient theorem we can see that $\rank(\Kh(L;\Z_2))= 20$. Batson-Seed's spectral sequence implies that at least one component of $L$ is unknotted, and the other is a knot $L_2$ with $\rank(\Kh(L_2;\Z_2))\leq10$. If $\rank(\Kh(L_2;\Z_2))\leq10$ then $\rank(\widetilde{\Kh}(L_2;\Z_2))\leq 5$, so that $K$ is an unknot, trefoil, figure eight or a knot with $\rank(\widetilde{\Kh}(L_2;\Z_2))=5$. Let $U$ denote the unknot. To see that $L_2$ must also be unknotted we proceed by cases:
\begin{enumerate}
	\item $L_2$ is $T(2,3)$. In this case: \begin{align*}\Kh(T(2,3);\Q)\otimes\Kh(U;\Q)&\cong(\Q_{-6}\oplus\Q_{-3}^2\oplus\Q_{-1})\otimes(\Q_{-1}\oplus\Q_1)\\&\cong\Q_{-7}\oplus\Q_{-5}\oplus\Q_{-4}^2\oplus\Q_{-2}^3\oplus\Q_0\end{align*}
		But $\rank(\Kh(T(2,10);\Q)$, is rank one in exactly one $i-j$ grading, a contradiction. 
	
	\item $L_2$ is $T(2,-3)$. Then \begin{align*}\Kh(T(2,-3);\Q)\otimes\Kh(U;\Q)&\cong(\Q_{6}\oplus\Q_{3}^2\oplus\Q_{1})\otimes(\Q_{1}\oplus\Q_{-1})\\&\cong \Q_{7}\oplus\Q_{5}\oplus\Q_{4}^2\oplus\Q_{2}^3\oplus\Q_0\end{align*}
	
	But $\Kh(T(2,10);\Q)$, is rank one in exactly one $i-j$ grading, a contradiction.

	\item $\rank(\widetilde{\Kh}(L_2;\Z_2))=5$. In this case we have that $\rank(\widetilde{\Kh}(L_2;\Q))\leq 5$, whence $\rank(\widehat{\HFK}(L_2;\Q))\leq 5$. If $\rank(\widehat{\HFK}(L_2;\Q)<5$, then $L_2$ is a trefoil or the unknot, cases which we have already dealt with. Thus we may take $\rank(\widehat{\HFK})(L_2;\Q)=5$. We show that $L_2$ is a figure eight or has the same knot Floer homology as $T(2,\pm5)$. Suppose $L_2$ is of genus $g>0$. Note that $\rank(\widehat{\HFK}(L_2;\Q,g))\leq 2$. Note also that since $\Kh(L;\Z)$ is Khovanov thin, $\Kh(L,p;\Q)$ is $\delta$-thin~\cite{baldwin2017khovanov}, whence $\widehat{\HFK}(L)$ is $\delta$-thin. We proceed by cases.\begin{enumerate}
	    \item Suppose $\rank(\widehat{\HFK}(L_2;\Q,g))= 2$. Then $\widehat{\CFK}(L_2;\Q)$ contains a differential of length $2g$ and one of length $g$. From Proposition~\ref{E2collapsed}, we see that this yields a summand of $(\widehat{\CFL}(L;\Q),\partial)$ of rank $8g$ and another of rank $4g$. Thus $\rank(\widehat{\HFL}(L;\Q))\geq 12g$. It immediately follows that $g=1$. The symmetry of $\widehat{\HFL}(L)$ implies that $\rank(\widehat{\HFL}(L))\geq 24$, a contradiction.
	    
	    \item Suppose $\rank(\widehat{\HFK}(L_2;\Q,g))= 1$. Consider $\widehat{\CFL}(L_2;\Q)$. If $g=1$, then $L_2$ is the figure eight knot. If $g>1$ there are two possible forms for $\widehat{\HFK}(L_2)$, as shown in Figure~\ref{rank5hfk}. To see this we need to use the fact that if $\rank(\widehat{\HFK}(L_2,g-1))=1$, then there cannot be both differentials from the top grading to the next to top grading and from the next to bottom grading to the bottom grading. This is the algebraic content of~\cite{baldwin_note_2018}.

					\begin{figure}[h]
    \centering
    \includegraphics[width=7cm]{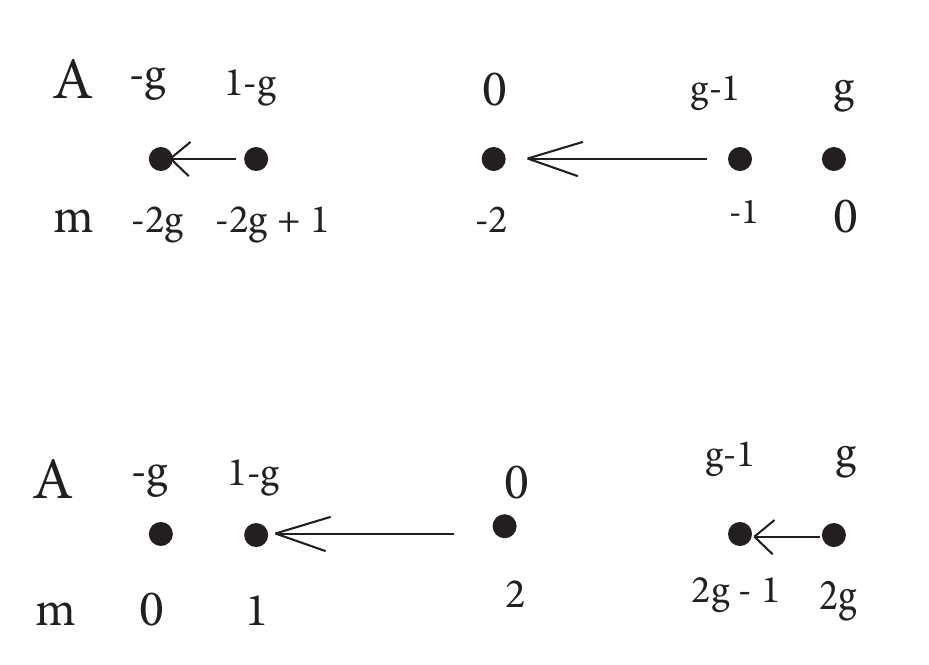}
    \caption{Different possibilities for $\widehat{\CFK}(L_2;\Q)$: here we denote the Alexander grading of the generators by $A$, and the Maslov grading by $m$.\label{rank5hfk}}
    \end{figure}
    
    In either case, since there is both a length one differential and a length $g-1$ differential, $\rank(\widehat{\HFL}(L;\Q))\geq 4(g-1)+4+4=4g+4$, whence $g\leq 4$. Suppose we are in the first case.  Indeed, since the minimal $A_2$ grading of these generators is $\frac{-3}{2}$, symmetry yields at least another $4(g-1)$ generators, so that $\rank(\widehat{\HFL}(L;\Q))\geq 8(g-1)+4+4+4$, so that $g=2$, and $\widehat{\HFK}(L_2;\Q)\cong \widehat{\HFK}(T(2,5);\Q)$.
	    
	    Suppose we are in the second case. We have that $g\in\{2,3,4\}$. If $g=2$ then $\widehat{\HFK}(L_2)\cong\widehat{\HFK}(T(2,-5))$. If $g=4$ using the symmetry of $\widehat{\HFL}(L;\Q)$ we find $\rank(\widehat{\HFL}(L;\Q))\geq 24$, a contradiction.
	    
	    If $g=3$, symmetry of $\widehat{\HFL}(L;\Q)$ implies that the $\rank(\widehat{\HFL}(L;\Q))\geq 24$, a contradiction.
	\end{enumerate}

	\item $L_2$ is the figure eight. In this case:
	\begin{align*}\Kh(U;\Q)\otimes\Kh(L_2;\Q)\cong(\Q_3\oplus\Q_1\oplus\Q_0^2\oplus\Q_{-1}\oplus\Q_{-3})\otimes(\Q_1\oplus\Q_{-1})\end{align*}
	
	But $\Kh(T(2,10);\Q)$, is rank one in exactly one $i-j$ grading, a contradiction.

	\item $\widehat{\HFK}(L_2;\Q)\cong\widehat{\HFK}(T(2,5);\Q)$, or $\widehat{\HFK}(L_2;\Q)\cong\widehat{\HFK}(T(2,-5);\Q)$. Since $\lk(L_1,L_2)=\pm 5$, there is a summand $\widehat{\HFK}(T(2,\pm5);\Q)\otimes V[\pm\frac{5}{2},a]$ of $\widehat{\HFL}(L;\Q)$, for some $a\in\pm\frac{5}{2}+\Z$. The symmetry of $\widehat{\HFL}(L;\Q)$ yields another $10$ generators. There can be no additional generators, as this would violate the rank inequality from Dowlin's spectral sequence -- $\rank(\widehat{\HFL}(L);\Q)\leq\rank(\Kh(T(2,10);\Q))=20$, where the final equality is obtained from the universal coefficient theorem. Note that in order that each $A_1$ grading have even rank, we must have that $\widehat{\HFL}(L;\Q)$ is supported only in $A_1$ gradings $\pm\frac{1}{2}$. However since $L_1$ is unknotted and there is a spectral sequence from $\widehat{\HFL}(L)$ to $\widehat{\HFL}(L_1)\otimes V[\frac{5}{2}]$ there must be generators in $A_1$ grading $\frac{5}{2}$, a contradiction.
\end{enumerate}
\end{proof}

	We can now proceed to the proofs of the Khovanov homology detection results. We apply Dowlin's spectral sequence and the classification of $E_2$ collapsed chain complexes given in Lemma~\ref{E2collapsed}. From here it suffices to prove the following pair of propositions.
	
	\begin{proposition}\label{combinedHFKKHT28detection}
	Let $F$ be $\Q$ or $\Z_2$. Suppose $L$ is a link satisfying the following conditions:
	\begin{enumerate}
	\item $L$ has exactly two components, $L_1$ and $L_2$, which are both unknots.
	\item $\lk(L_1,L_2)=4$
	\item $\chi(L)\leq-6$
	    \item $\rank(\widehat{\HFK}(L;F))\leq 16$.
	    \item $\widehat{\HFK}(L;F)$ is $\delta$-thin.
	\end{enumerate}
	
	Then $L$ is $T(2,8)$.
	\end{proposition}
	
		\begin{proposition}\label{combinedHFKKHT210detection}
	Let $F$ be $\Q$ or $\Z_2$. Suppose $L$ is a link satisfying the following conditions:
	\begin{enumerate}
	\item $L$ has two exactly two components, $L_1$ and $L_2$, which are both unknots.
	\item $\lk(L_1,L_2)=5$
	\item $\chi(L)\leq-8$
	    \item $\rank(\widehat{\HFK}(L;F))\leq 20$.
	    \item $\widehat{\HFK}(L;F)$ is $\delta$-thin.
	\end{enumerate}
	
	Then $L$ is $T(2,10)$.
	\end{proposition}
	
Before proving these propositions, we apply them to obtain Theorem's~\ref{KHTT28} and~\ref{KHT210} respectively.

	\begin{proof}[Proof of Theorem~\ref{KHTT28}]
	It suffices to show that $L$ satisfies the conditions in Proposition~\ref{combinedHFKKHT28detection}. Conditions 1,2 and 3 are implied by Lemmas~\ref{KHcomponentsgenuslinking} and~\ref{KHt228components}. To see that condition $4$ holds, observe that $\rank(\Kh(L;\Z_2)=16$ by the universal coefficient theorem. It follows that $\rank(\widetilde{\Kh}(L;\Z_2))=8$ by~\cite{shumakovitch2014torsion}, so that $\rank(\widetilde{\Kh}(L;\Q))\leq 8$ by the universal coefficient theorem. $\rank(\Kh(L,p,\Q))\leq\rank(\widetilde{\Kh}(L;\Q))$ by \cite[Theorem 1.2]{baldwin2017khovanov}, whence $\widehat{\HFK}(L;\Q)\leq 16$.

	To see that condition 5 holds, note that $\Kh(L;\Q)$ is $\delta$-thin, so that $\Kh(L,p;\Q)$ is $\delta$-thin by \cite[Lemma 2.11]{baldwin2017khovanov}.
	\end{proof}
	
	\begin{proof}[Proof of Theorem~\ref{KHT210}]
	    This follows by the same argument as in the proof of Theorem~\ref{KHTT28}.
	\end{proof}
	
We now prove the two requisite lemmas~\ref{combinedHFKKHT28detection}, and ~\ref{combinedHFKKHT210detection}.

		\begin{proof}[Proof of Proposition~\ref{combinedHFKKHT28detection}]
	Suppose $L$ satisfies the conditions given in the Proposition. Consider $\widehat{\CFL}(L;\Q)$. Since $\widehat{\HFK}(L)$ is $\delta$-thin, $\widehat{\CFL}(L;\Q)$ is an $E_2$ collapsed chain complex. Since $\chi(L)\leq -6$ the maximal Alexander grading of $\widehat{\HFK}(L;\Q)$ is at least $3$. 
	
	Consider $\widehat{\CFL}(L)$. Since $L$ has unknotted components, it follows that $(\widehat{\CFL}(L;\Q),\partial)$ cannot contain a $V^l_d$ or $H_d^l$ summand, and that there are Maslov grading $0$ and $-1$ generators with $A_i=2$. Suppose the Maslov grading $0$ generator is of $A_1+A_2$-grading $l$.
 
 Suppose $l>4$. Since $L_1,L_2$ are unknots with linking number $4$, there must be Maslov grading $0$ and Maslov grading $-1$ generators on both the line $A_1=2$ and the line $A_2=2$. In particular there are Maslov grading $0$ generators of bi-Alexander gradings $(2,l-2)$ and $(l-2,2)$ and Maslov grading $-1$ generators of bi-Alexander gradings $(2,l-3)$, $(l-3,2)$. Since these persist under the spectral sequence to $V$, and are suitably identified there, it follows that $(\widehat{\CFL}(L),\partial)$ has summands $X_{-1}^{l-5}[2,2]\oplus X_0^{l-4}[2,2]$. From the symmetry of $\widehat{\HFL}(L;\Q)$, and the fact that $\rank(\widehat{\HFL}(L;\Q))\leq 16$ we see immediately that $l\leq 6$.

					If $l\leq 4$ then there are summands $Y_{0}^{4-l}[l-2,l-2]\oplus Y_{-1}^{5-l}[l-3,l-3]$. Note $l\geq 1$ as else $\rank(\widehat{\HFK}(L))>16$. Similarly $l=2$ would violate the rank bound after accounting for the symmetry of $\widehat{\HFL}(L)$. If $l=1$ then $\rank(\widehat{\HFL}(L))\geq 16$. It follows that there are no additional generators and we immediately see that $L$ bounds an annulus, contradicting the fact that the maximal Alexander grading is at least $3$.		
				If $l\in \{3, 6\}$, then the complex is determined by the symmetry of $\widehat{\HFL}(L;\Q)$. In particular we see that there are copies of $B_d$ in specified bi-gradings. In the $l=6$ case we quickly see that each component is a braid axis by~\cite[Proposition 1]{martin2022khovanov}, which implies that the maximal $A_i$ gradings are $2$, ruling out the $l=6$ case. If $l=3$ we see immediately that $\chi(L)>-6$, a contradiction.
				
				If $l=5$ we readily see that -- apart from in the case that $\widehat{\CFL}(L;\Q)$ contains summands either $B_{-6}[-3,2]\oplus B_{-6}[2,-3]$ or $B_{-1}[2,2]\oplus B_{-11}[-3,-3]$ -- at least one component is a braid axis by~\cite[Propositon 1]{martin2022khovanov}, meaning that the maximal $A_i$ grading is $2$, a contradiction. The $B_{-1}[2,2]\oplus B_{-11}[-3,-3]$ is excluded by noting that upon reversing the orientation of a component of $L$, $L$ bounds an annulus and hence $L$ is $T(2,8)$, a contradiction since it has the wrong link Floer homology type.

    $B_{-6}[-3,2]\oplus B_{-6}[2,-3]$ case can be excluded using Proposition~\ref{prop:algebraicni}, with $x$, $\bar{x}$ and $y$ generators of Alexander bi-gradings $(-3,3),(3,-3)$ and $(3,-2)$ respectively.

    If $l=4$ then by symmetry we have a summand $B_{-8}[-2,-2]$. There are at most two remaining summands. Suppose there are two. Let $B_{i+j-4}[i,j]$ be such a summand. Observe that that $-2\leq i,j\leq 1$, as else the rank of $\widehat{\HFL}(L)$ in the maximal $A_1$ or $A_2$ grading is $2$ but the maximal such grading is not $2$. We may assume without loss of generality that $j\geq i$. If $(i,j)\in\{(-2,0),(-1,0),(-2,1),(-1,1),(0,1)\}$ by applying Proposition~\ref{prop:algebraicni} with $x$ the generator of grading $(i,j+1)$, $\bar{x}$ the generator of grading $(-i,-j-1)$ and $y$ the generator of grading $(-i,-j)$. The cases that $(i,j)=(-2,-1)$ can be excluded by considering the link $L'$ given be reordering the two components of $L$, which reduces this situation to the case $(i,j)=(0,1)$. We can thus conclude that $i=j$. It follows that, upon reversing the orientation of a component, $L$ bounds an annulus. Since $L$ has unknotted components and linking number $4$ it follows that $L$ is $T(2,8)$, as desired.

    If there is only one remaining $B_d$ summand, then it must be of the form $B_d[-1/2,-1/2]$ for some $d$ by the symmetry of Link Floer homology. But this is a contradiction, since $\lk(L)$ is even, so that the Alexander gradings are $\Z$ valued. 

     If there is no other $B_d$ summand, then again we see that upon reversing the orientation of a component of $L$, $L$ bounds an annulus and hence is $T(2,8)$, a contradiction since it has the wrong link Floer homology type.

				\end{proof}

The proof of Proposition~\ref{combinedHFKKHT210detection} is virtually the same, but with some additional case analysis.

\begin{proof}[Proof of Proposition~\ref{combinedHFKKHT210detection}]
	 Suppose $L$ is as in the hypotheses of the Proposition. Since $L$ is delta-thin, $\widehat{\HFL}(T(2,10);\Q)$ is $E_2$-collapsed. Since $\chi(L)\leq -8$ the maximal Alexander grading is at least $n$. Since the components of $L$ are unknotted $\widehat{\CFL}(L;\Q)$ consists only of $X^l_d,Y^l_d$ and $B_d$ summands. Note that $L$ has Maslov index $0$, $-1$ generators in each line $A_i=\frac{5}{2}$. Let $l$ be the $A=A_1+A_2$ grading of the Maslov index $0$ generator. If $l> 5$ then $\widehat{\CFL}(L;\Q)$ has summands $X_{-1}^{l-6}[\dfrac{5}{2},\dfrac{5}{2}]\oplus X_{0}^{l-5}[\frac{5}{2},\frac{5}{2}]$. Note that for $l> 5$, $\rank(X_{-1}^{l-5}[\dfrac{5}{2},\dfrac{5}{2}]\oplus X_{0}^{l-4}[\dfrac{5}{2},\dfrac{5}{2}])=4l-16$. Observe that the symmetry of $\widehat{\HFL}(L;\Q)$ together with the fact that $\rank(\widehat{\HFL}(L;\Q))\leq 20$, imply that $l\leq 7$.

 If $l=7$ then the position of the remaining $B_d$ summands are forced by the symmetry of $\widehat{\HFL}(L)$. We find that the maximal $A_i$ gradings are $\frac{9}{2}$, and each of rank $2$. It follows that each component is braided with respect to the other. But then, since each component should be unknotted, it follows that the maximal Alexander grading should be half the linking number of $L$, i.e. $\frac{5}{2}$, a contradiction.

	If $l=6$, then either $\widehat{\CFL}(L)$ contains a summand $ X_{-1}^{0}[\dfrac{5}{2},\dfrac{5}{2}]\oplus X_{0}^{1}[\dfrac{5}{2},\dfrac{5}{2}]\oplus B_{-13}[\frac{-7}{2},\frac{-7}{2}]$ and at most three extra $B_d$ summands.
 
 If there is one extra $B_d$ summands then must be of the form $B_{-7}[\frac{-1}{2},\frac{-1}{2}]$ by symmetry, so if there are either zero or one extra $B_d$ summands we see that $L$ bounds an annulus upon reversing the orientation of a component, and $L$ is therefore should be $T(2,10)$, a contradiction since $T(2,10)$ has distinct link Floer homology.
 
 If $\widehat{\HFL}(L;\Q)$ contains two extra $B_d$ summands, then unless they are of the form $B_{-1}[\frac{5}{2},\frac{5}{2}]\oplus B_{-13}[\frac{-7}{2},\frac{-7}{2}]$ or $B_{-7}[\frac{5}{2},\frac{-7}{2}]\oplus B_{-7}[\frac{-7}{2},\frac{5}{2}]$, we see immediately that at least one component is a braid axis, and hence the maximal $A_i$ grading is $\dfrac{5}{2}$, a contradiction. The $B_{-1}[\frac{5}{2},\frac{5}{2}]\oplus B_{-13}[\frac{-7}{2},\frac{-7}{2}]$ is excluded by noting that $L$ bounds an annulus upon reversing the orientation of a component whence $L$ is $T(2,10)$, a contradiction since $T(2,10)$ has distinct link Floer homology. The $B_{-7}[\frac{5}{2},\frac{-7}{2}]\oplus B_{-7}[\frac{-7}{2},\frac{5}{2}]$ case is excluded by noting that we can apply Proposition~\ref{prop:algebraicni} with $x$ a generator in $(i,j)$ grading $(\frac{-7}{2},\frac{7}{2})$, $\bar{x}$ a generator of grading $(\frac{7}{2},\frac{-7}{2})$ and $y$ a generator of grading $(\frac{7}{2},\frac{-5}{2})$.

 If there are three additional $B_d$ summands then we note that in order to satisfy the symmetry properties there must be a summand $B_{-7}[\frac{-1}{2},\frac{-1}{2}]$. Unless the remaining $B_d$ summands are $B_{-1}[\frac{5}{2},\frac{5}{2}]\oplus B_{-13}[\frac{-7}{2},\frac{-7}{2}]$, or $B_{-8}[-\frac{5}{2},\frac{7}{2}]$ $B_{-8}[\frac{3}{2},\frac{-5}{2}]$, we have that at least one component, $L_i$ is a braid axis, so again the maximal $A_i$ grading must be $\frac{5}{2}$, a contradiction. In the first case, we see that $L$ bounds an annulus upon reversing its orientation, a contradiction, since no link $T(2,2n)$ has the requisite homology. We can exclude the second case by an application of Proposition~\ref{prop:algebraicni} with $x$ a generator in $(i,j)$ grading $(\frac{-7}{2},\frac{7}{2})$, $\bar{x}$ a generator of grading $(\frac{7}{2},\frac{-7}{2})$ and $y$ a generator of grading $(\frac{7}{2},\frac{-5}{2})$.
	
	If $l\leq 5$ then there are summands $Y_0^{5-l}[l-\frac{5}{2},l-\frac{5}{2}]\oplus Y_{-1}^{6-l}[l-\frac{7}{2},l-\frac{7}{2}]$. Considering the symmetry properties of $\widehat{\HFL}(L)$ we readily see that either $l=0$ or $l\geq 4$, as else $\rank(\widehat{\HFK}(L;\Q))>20$. If $l=0$ then there are no additional generators, and we see that $L$ bounds an annulus, contradicting our assumption on the Euler characteristic. If $l=4$ then $\widehat{\HFL}(L;\Q)$ either contains no extra generators or a cube $B_{-6}[\dfrac{-1}{2},\dfrac{-1}{2}]$ by symmetry. But in this case the maximal Alexander grading of $\widehat{\HFK}(L)$ is $4$, contradicting the assumption that $\chi(L)\leq -8$.

We are thus left only with the case that $l=5$, in which case $\widehat{\CFL}(L;\Q)$ has a summand $X_0[\frac{5}{2},\frac{5}{2}]\oplus Y_{-1}[\frac{3}{2},\frac{3}{2}]$. By symmetry $\widehat{\CFL}(L;\Q)$ has a summand $B_{-10}[\frac{-5}{2},\frac{-5}{2}]$ $\widehat{\CFL}(L;\Q)$ can have at most three other $B_d$ summands. If there are one or three more $B_d$ summands then the symmetry of $\widehat{\HFL}(L;\Q)$ implies that one such summand must be of the form $B_{-6}[\frac{-1}{2},\frac{-1}{2}]$. Thus if $\widehat{\HFL}(L;\Q)$ has $0$ or $1$ extra $B_d$ summands we see that after reversing the orientation of a component of $L$, $L$ bounds an annulus, and is therefore $T(2,10)$, a contradiction since $T(2,10)$ has distinct link Floer homology.

Suppose $\widehat{\CFL}(L;\Q)$ has two or three additional $B_d$ summands. If there are three, then the symmetry of $\widehat{\HFL}(L;\Q)$ implies again that $\widehat{\CFL}(L;\Q)$ has a $B_{-6}[\frac{-1}{2},\frac{-1}{2}]$ summand. Let $B_{i+j-5}[i,j]$ be one of the two remaining summands. Without loss of generality we may take $j\geq i$. Observe that unless $-\frac{5}{2}\leq i\leq j\leq\frac{3}{2}$, we have that $\widehat{\HFL}(L)$ is of rank $2$ in one of the two maximal Alexander gradings, $A_i$, so that $L_i$ is a braid axis, but the maximal $A_i$ grading is not $\frac{5}{2}$, a contradiction.

Suppose $(i,j)\not\in\{(\frac{-5}{2},\frac{-5}{2}),(\frac{-5}{2},\frac{-3}{2}),(\frac{-3}{2},\frac{-3}{2}),(\frac{-1}{2},\frac{-1}{2}),(\frac{-1}{2},\frac{1}{2}),(\frac{1}{2},\frac{1}{2}), (\frac{3}{2},\frac{3}{2})\}$. Then we can apply Proposition~\ref{prop:algebraicni} with  with $x$ the generator of grading $(i,j+1)$, $\bar{x}$ the generator of grading $(-i,-j-1)$ and $y$ the generator of grading $(-i,-j)$ to obtain a contradiction.

If $(i,j)=(\frac{-5}{2},\frac{-3}{2})$ or $(\frac{-1}{2},\frac{1}{2})$ then we can consider the link $L'$ obtained by switching the order of the two components. This reduces these two cases to the case $(i,j)=(\frac{1}{2},\frac{3}{2})$ and $(i,j)=(\-\frac{3}{2},-\frac{1}{2})$ that were excluded above. It follows that $\widehat{\HFL}(L)$ is contained in the lines of $A_2=A_1$, $A_2=A_1\pm\frac{1}{2}$. It follows in turn that $L$ bounds an annulus after changing the orientation of a component. Since $L$ has unknotted components and linking number $5$ it follows that $L$ is $T(2,10)$, as required.
\end{proof}

We conclude by noting that Martin's detection result for $T(2,6)$~\cite{martin2022khovanov}, $T(2,4)$ detection~\cite{xie2022links} and $T(2,2)$ detection~\cite{baldwin2019khovanov} could have been obtained by similar methods to those presented in this paper.

\end{subsection}
\end{section}

\bibliographystyle{plain}
\bibliography{main}

\end{document}